\documentclass[10pt]{article}
\pdfoutput=1
\usepackage{hyperref}
\hypersetup{hypertexnames = false, bookmarksdepth = 2, bookmarksopen = true, colorlinks, linkcolor = black, citecolor = black, urlcolor = black, pdfstartview={XYZ null null 1}}

\usepackage{amsfonts}
\usepackage[fleqn, leqno]{amsmath}
\usepackage{amsthm}

\usepackage[capitalise,noabbrev]{cleveref}

\usepackage{tocloft}

\usepackage{booktabs}
\usepackage{colonequals}
\usepackage{enumitem}
\usepackage{fullpage}
\usepackage{mathtools}
\usepackage{parskip}
\usepackage{rotating}
\usepackage{subcaption}
\usepackage{thmtools}
\usepackage{tikz-cd}
\usepackage[colorinlistoftodos, textsize = footnotesize]{todonotes}
\usepackage{xparse}
\usepackage{xpatch}
\usepackage{xspace}

\usepackage[utf8]{inputenc}
\usepackage[T1]{fontenc}
\usepackage[scaled=0.83]{beramono}
\usepackage{eucal}
\usepackage{mathrsfs}
\usepackage{microtype}
\usepackage{gitinfo2}

\usetikzlibrary{intersections}

\usepackage[backend=biber, maxbibnames=99, datamodel=mrnumber, sortcites]{biblatex}

\DeclareFieldFormat{mrnumber}{%
  MR\addcolon\space
  \ifhyperref
    {\href{http://www.ams.org/mathscinet-getitem?mr=#1}{\nolinkurl{#1}}}
    {\nolinkurl{#1}}}

\renewbibmacro*{doi+eprint+url}{%
  \iftoggle{bbx:doi}
    {\printfield{doi}}
    {}%
  \newunit\newblock
  \printfield{mrnumber}%
  \newunit\newblock
  \iftoggle{bbx:eprint}
    {\usebibmacro{eprint}}
    {}%
  \newunit\newblock
  \iftoggle{bbx:url}
    {\usebibmacro{url+urldate}}
    {}}

\DeclareSourcemap{
  \maps[datatype=bibtex]{
    \map{
      \step[fieldset=issn, null]
    }
  }
}
\xpretobibmacro{bib+doi+url}
  {\iffieldundef{doi}
     {}
     {\clearfield{url}}}
  {}{}

\xpretobibmacro{cite+doi+url}
  {\iffieldundef{doi}
     {}
     {\clearfield{url}}}
  {}{}

\newcounter{todocounter}
\DeclareDocumentCommand\addreference{g}{\stepcounter{todocounter}\todo[color = blue!30]{\thetodocounter. Add reference\IfNoValueF{#1}{: #1}}\xspace}
\DeclareDocumentCommand\checkthis{g}{\stepcounter{todocounter}\todo[color = red!50]{\thetodocounter. Check this\IfNoValueF{#1}{: #1}}\xspace}
\DeclareDocumentCommand\fixthis{g}{\stepcounter{todocounter}\todo[color = orange!50]{\thetodocounter. Fix this\IfNoValueF{#1}{: #1}}\xspace}
\DeclareDocumentCommand\expand{g}{\stepcounter{todocounter}\todo[color = green!50]{\thetodocounter. Expand\IfNoValueF{#1}{: #1}}\xspace}

\declaretheoremstyle[spaceabove = 3pt, spacebelow = 3pt, bodyfont = \itshape]{theorem}
\declaretheoremstyle[spaceabove = 3pt, spacebelow = 3pt]{remark}

\declaretheorem[style=theorem]{theorem}
\declaretheorem[style=theorem, sibling=theorem]{corollary}
\declaretheorem[style=theorem, sibling=theorem]{lemma}
\declaretheorem[style=theorem, sibling=theorem]{proposition}

\declaretheorem[style=remark, sibling=theorem]{definition}
\declaretheorem[style=remark, sibling=theorem]{example}
\declaretheorem[style=remark, sibling=theorem]{remark}

\declaretheorem[style=theorem, numberwithin=section, title=Theorem]{alphatheorem}

\crefname{alphatheorem}{Theorem}{Theorems}
\crefname{alphaconjecture}{Conjecture}{Conjectures}
\crefname{alphacorollary}{Corollary}{Corollaries}
\crefname{alphaproposition}{Proposition}{Propositions}

\crefformat{enumi}{#2\textup{(#1)}#3}

\setlist[enumerate]{font=\normalfont}

\mathchardef\mhyphen="2D

\newcommand\order[1]{\mathscr{#1}}
\newcommand\stack[1]{\mathcal{#1}}

\newcommand\bounded{\ensuremath{\mathrm{b}}}
\newcommand\can{\ensuremath{\mathrm{can}}}
\newcommand\field{\mathbf{k}}
\newcommand\Gm{\ensuremath{\mathbb{G}_{\mathrm{m}}}}
\newcommand\rt{\ensuremath{\mathrm{root}}} 
\newcommand\sing{\ensuremath{\mathrm{sing}}}
\newcommand\smooth{\ensuremath{\mathrm{sm}}}

\DeclareMathOperator\Cl{C\ell}
\DeclareMathOperator\codiff{codiff}
\DeclareMathOperator\coh{coh}
\DeclareMathOperator\derived{\mathbf{D}}

\DeclareMathOperator\distinguished{D}
\DeclareMathOperator\End{End}
\DeclareMathOperator\Ext{Ext}
\DeclareMathOperator\fd{fd}
\DeclareMathOperator\Frac{Frac}
\DeclareMathOperator\gldim{gl\,dim}
\DeclareMathOperator\gr{gr}
\DeclareMathOperator\HH{H}
\DeclareMathOperator\Hom{Hom}
\DeclareMathOperator\id{id}
\DeclareMathOperator\intHom{\mathcal{H}\kern -.5pt om}
\DeclareMathOperator\Mat{Mat}
\DeclareMathOperator\mmod{mod}
\DeclareMathOperator\mult{mult}
\DeclareMathOperator\Proj{Proj}
\DeclareMathOperator\Qcoh{Qcoh}
\DeclareMathOperator\qgr{qgr}
\DeclareMathOperator\rad{rad}
\DeclareMathOperator\ramification{D}
\DeclareMathOperator\relSpec{\underline{Spec}}
\DeclareMathOperator\rk{rk}
\DeclareMathOperator\Spec{Spec}
\DeclareMathOperator\trd{trd}
\DeclareMathOperator\VV{\mathbb{V}}
\DeclareMathOperator\ZZ{Z}

\newsavebox{\pullback}
\sbox\pullback{%
\begin{tikzpicture}%
\draw (0,0) -- (1ex,0ex) -- (1ex,1ex);%
\end{tikzpicture}}

\newcommand\bbP{\mathbb{P}}
\newcommand\bbZ{\mathbb{Z}}

\newcommand\cM{\mathcal{M}}
\newcommand\cN{\mathcal{N}}
\newcommand\cO{\mathcal{O}}

\title{Central curves on noncommutative surfaces}
\author{Thilo Baumann, Pieter Belmans and Okke van Garderen}

\begin{document}

\maketitle

\begin{abstract}
  There exists a dictionary between
  hereditary orders and smooth stacky curves,
  resp.~tame orders of global dimension~2 and Azumaya algebras on smooth stacky surfaces.
  We extend this dictionary
  by explaining how the restriction of a tame order
  to a curve on the underlying surface
  corresponds to the fiber product of the curve with the stacky surface.
  By considering ``bad'' intersections
  we can start extending the dictionary in the 1-dimensional case
  to include non-hereditary orders
  and singular stacky curves.
  Two applications of these results are
  a novel description and classification of noncommutative conics
  in graded Clifford algebras,
  giving a geometric proof of results of Hu--Matsuno--Mori,
  and a complete understanding and classification of skew cubics,
  generalizing the work of Kanazawa for Fermat skew cubics.
\end{abstract}

\tableofcontents

\section{Introduction}
There are rich interactions between noncommutative algebraic geometry
and algebraic stacks.
First of all,
there exists a one-to-one correspondence between
\begin{enumerate}
  \item hereditary orders on smooth (and separated) curves;
  \item smooth (and separated) Deligne--Mumford stacks which are generically curves
\end{enumerate}
by \cite{MR2018958},
giving rise to an equivalence of categories between
the category of coherent sheaves of modules for a hereditary order
and the category of coherent sheaves on the associated stacky curve
(and vice versa).
The curve underlying the hereditary order
is the coarse moduli space of the Deligne--Mumford stack,
and the Deligne--Mumford stack admits
a bottom-up description in terms of root stacks \cite{MR3719470},
and thus we completely understand their categories of sheaves.

Recently, this correspondence was extended to a dictionary between
\begin{enumerate}
  \item tame orders of global dimension~2, and
  \item Azumaya algebras on certain smooth stacky surfaces.
\end{enumerate}
Again it takes the form of an equivalence of categories between
the category of coherent sheaves of modules for the tame order~$\order{A}$ on the surface~$S$
and the category of coherent sheaves of modules for an Azumaya algebra~$\order{B}$
on an associated stacky surface~$\stack{S}$ \cite{2206.13359}.
The surface~$S$ underlying the tame order
is the coarse moduli space of~$\stack{S}$.

\paragraph{Restricting the dictionary to central curves}
In this paper we explain what happens if we take a curve~$C$
on the underlying surface~$S$,
the \emph{central curve},
and~restrict the tame order~$\order{A}$ on~$S$
to a sheaf of algebras~$\order{A}|_C$ on~$C$.

On the other side of the dictionary,
we can take the fiber product
\begin{equation}
  \label{equation:fiber-product-introduction}
  \begin{tikzcd}
    \stack{C} \arrow[r, hook] \arrow[d] & \stack{S} \arrow[d] \\
    C \arrow[r, hook] & S
  \end{tikzcd}
\end{equation}
and restrict the Azumaya algebra~$\order{B}$ on~$\stack{S}$ to an Azumaya algebra on~$\stack{C}$.

Our first result is that this restriction is well-behaved
with respect to the equivalence from \cite{2206.13359}.
\begin{alphatheorem}
  \label{theorem:restriction-introduction}
  Let~$\order{A}$ be a tame order of global dimension~2
  on a smooth quasiprojective surface $S$,
  and let~$\order{B}$ be the Azumaya algebra on the stack~$\stack{S}$
  for which~$\coh(S,\order{A})\simeq\coh(\stack{S},\order{B})$.
  Given a curve~$C\subset S$ with~$\stack{C}$ the pullback to the stack as in \eqref{equation:fiber-product-introduction},
  there exists an equivalence
  \begin{equation}
    \coh(C,\order{A}|_C)
    \simeq
    \coh(\stack{C},\order{B}|_{\stack{C}}).
  \end{equation}
\end{alphatheorem}
The more precise version is given in \cref{theorem:restriction}.

We want a more precise understanding of the sheaf of algebras~$\order{A}|_C$.
Denoting~$\Delta\subset S$ the discriminant of~$\order{A}$,
we have the following.
\begin{alphatheorem}
  \label{theorem:restriction-properties-introduction}
  Let~$S,\mathcal{A},C$ be as in \cref{theorem:restriction-introduction},
  and assume moreover that~$C$ is integral.
  Then:
  \begin{enumerate}
    \item\label{enumerate:restriction-order}
      The sheaf of algebras~$\order{A}|_C$ is an order
      if and only if~$C$ is not contained in~$\Delta$.
  \end{enumerate}
  Assume that~$\order{A}|_C$ is an order. Then:
  \begin{enumerate}
      \setcounter{enumi}{1}
    \item\label{enumerate:restriction-azumaya}
      $\order{A}|_C$ is Azumaya
      if and only if~$C\cap\Delta=\emptyset$.
    \item\label{enumerate:restriction-hereditary}
      If~$C$ is smooth and intersects~$\Delta$ transversely in its smooth locus,
      then~$\order{A}|_C$ is hereditary.
  \end{enumerate}
\end{alphatheorem}
The proof, and some extensions of this result,
is given in \cref{subsection:properties-restriction}.

\paragraph{Application: noncommutative plane curves}
The restriction results in \cref{theorem:restriction-introduction,theorem:restriction-properties-introduction}
can be used to
provide new geometric proofs
and significant extensions
of several results in the literature,
obtained by vastly different methods,
for objects that can reasonably be called ``noncommutative plane curves''.

Namely, consider
a quadratic 3-dimensional Artin--Schelter regular algebra~$A$,
so that~$\qgr A$ is a noncommutative projective plane.
Let~$f\in\ZZ(A)_d$ be a central element of degree~$d$.
The noncommutative projective variety~$\qgr A/(f)$
can be considered as a noncommutative plane curve.
This construction is worked out for degree-2 elements in \cite{MR4531545,MR4575408},
and in a more specialised situation for degree-3 elements in \cite{MR3313507}.
The methods in op.~cit.~do not generalize to other settings.

However,
the Artin--Schelter regular algebras~$A$ in all of the cited papers
are in fact finite over their centers\footnote{
  Except for one exceptional case:
  by \cite[Theorem~3.6]{MR4531545},
  the only 3-dimensional Calabi--Yau Artin--Schelter regular algebra
  with a central element of degree~2 which is not a graded Clifford algebra or the commutative polynomial algebra,
  is~$\field\langle x,y,z\rangle/(yz-zy+x^2,zx-xz,xy-yx)$,
  the quantization of the Weyl algebra,
  which only has the central element~$x^2$.
  We will ignore this case.}.
Thus, to such an algebra~$A$
we can apply the central Proj construction of \cite{MR1356364} to obtain an order $\order{A}$ on an associated~$\mathbb{P}^2$.
The element~$f$ defines a central curve~$C\subset \mathbb{P}^2$ and there is an equivalence
\begin{equation}
  \qgr A/(f) \simeq \coh(C,\order{A}|_C).
\end{equation}
This allows us to apply the dictionary,
to give a unified treatment of the noncommutative conics and cubics,
which we describe below.
This approach moreover generalizes easily to other settings,
we will discuss two interesting phenomena in \cref{example:Clifford-degree-4,example:exotic-del-pezzo-order}.

Of course, it is classical that any quadratic Artin--Schelter regular algebra
admits a central element~$g\in\ZZ(A)_3$ of degree~3.
The noncommutative plane curve~$\qgr A/(g)$
is the plane cubic curve~$C$ appearing in the classification of Artin--Schelter regular algebras,
with~$A/(g)$ being the twisted homogeneous coordinate ring of~$C$.
Generically, the center of~$A$ is generated by~$g$,
meaning that no other interesting noncommutative plane curves exist.
Asking that~$A$ is finite over its center is therefore
a natural condition to ensure the existence of noncommutative plane curves.

\paragraph{Noncommutative conics in graded Clifford algebras}
We will revisit the degree-2 case,
studied and classified in \cite{MR4531545,MR4575408},
from the perspective of central curves.
The ambient Artin--Schelter regular algebras we will consider
are graded Clifford algebras,
see \cref{subsection:graded-clifford-algebras} for more details,
and these admit a 3-dimensional linear system of noncommutative conics.
A central element~$f\in\ZZ(A)_2$
defines a plane curve of degree~1
in the associated~$\mathbb{P}^2$
coming from the central Proj construction.

\begin{alphatheorem}
  \label{theorem:classification-of-clifford-conics}
  There are~6 isomorphism classes of noncommutative conics
  within 3-dimensional graded Clifford algebras.
  Their algebraic resp.~stacky properties
  are summarized in \cref{table:local-properties-clifford-conic-orders,table:properties-stacky-clifford-conics}.
\end{alphatheorem}

We refer to \cref{theorem:clifford-conic-orders-classification,theorem:clifford-conic-stacks,theorem:classification}
for the precise version of this theorem.

\Cref{table:local-properties-clifford-conic-orders,table:properties-stacky-clifford-conics}
suggest how to extend the dictionary for hereditary orders and root stacks
beyond the smooth case.
For~5 out of~6 isomorphism classes in \cref{theorem:classification-of-clifford-conics},
the abelian category~$\qgr A/(f)$ has infinite global dimension,
so that the sheaf of algebras~$\order{A}|_{\mathbb{P}^1}$ is no longer a hereditary order,
resp.~the Deligne--Mumford stack~$\stack{C}$ from \cref{theorem:restriction-introduction}
is no longer smooth.

E.g.,
the property of being a \emph{tiled} order
is seen to correspond to being a \emph{root stack},
possibly in a non-reduced divisor,
at least in the setting of noncommutative conics.
Understanding the precise shape of the dictionary,
relating order-theoretic properties
to stacky properties,
is beyond the scope of this article,
and left for future work.

\paragraph{Noncommutative (Fermat) cubics in skew polynomial algebras}
For a second application
we will consider the noncommutative Fermat cubics in 3-dimensional skew polynomial algebras
studied by Kanazawa~\cite{MR3313507}.
Op.~cit.~considers the skew polynomial algebra~$A=\field_q[x,y,z]$, where
\begin{equation}
  q=\begin{pmatrix}
    1 & q_{1,2} & q_{1,3} \\
    q_{1,2}^{-1} & 1 & q_{2,3} \\
    q_{1,3}^{-1} & q_{2,3}^{-1} & 1
  \end{pmatrix}
\end{equation}
and the~$q_{i,j}$ are cube roots of unity;
with central element~$f=x^3+y^3+z^3$ the Fermat cubic,
and~$B\colonequals A/(f)$ the homogeneous coordinate ring of the noncommutative plane curve.
Then \cite[Theorem~1.1]{MR3313507} (for~$n=3$) states that
the category~$\qgr B$ is 1-Calabi--Yau if and only if
\begin{equation}
  \label{equation:kanazawa-independence}
  q_{1,2}^{-1}q_{1,3}^{-1}=q_{1,2}q_{2,3}^{-1}=q_{1,3}q_{2,3}.
\end{equation}
We will re-examine this case in \cref{subsection:kanazawa-is-fermat}
using our machinery,
and explain how we actually obtain an equivalence~$\qgr B\simeq\coh E$
for the Fermat elliptic curve~$E$,
which is not mentioned explicitly in \cite{MR3313507}.

The benefit of our approach is that our methods in fact work for:
\begin{enumerate}
  \item \emph{every} element in the 3-dimensional linear system (referred to as net) of central cubics~$\langle x^3,y^3,z^3\rangle$;
  \item the case where \eqref{equation:kanazawa-independence} \emph{does not hold}.
\end{enumerate}
In the latter case,
by \cite[Theorem~1.1]{MR3313507}~$\qgr B$ is \emph{not} 1-Calabi--Yau.
However, op.~cit.~shows that it has global dimension~1 for the Fermat cubic,
and one can moreover deduce from \cite[Proposition~2.4]{MR3313507}
that its Serre functor is~3-torsion.
In \cref{subsection:non-kanazawa-is-tubular}
we explain
how~$\qgr B$ is the (up to isomorphism unique) tubular weighted projective line of type~$(3,3,3)$,
which is indeed fractional Calabi--Yau of dimension~$3/3$,
and we extend the result to any central cubic in the linear system~$\langle x^3,y^3,z^3\rangle$.

We refer to \cref{proposition:kanazawa-fermat-cubic-1,proposition:kanazawa-fermat-cubic-2}
for the precise statements in case~\eqref{equation:kanazawa-independence} holds
(resp.~does not hold).

\paragraph{Notation and conventions}
Throughout we work over an algebraically closed field $\field$ of characteristic $0$.

\paragraph{Acknowledgements}
T.B.~was partially supported by the Luxembourg National Research Fund (PRIDE R-AGR-3528-16-Z).
P.B.~was partially supported by the Luxembourg National Research Fund (FNR--17113194).

\section{Dictionaries between orders and stacks}
\label{section:dictionaries}
In this section we survey the dictionary between orders and stacks.
For hereditary orders on curves this is done in \cref{subsection:dictionary-dim-1},
for tame orders on surfaces this is done in \cref{subsection:dictionary-dim-2}.
We also recall some of the details of the proofs in \cite{2206.13359}
as they are necessary for the proofs in this paper.
To highlight the role that generalized Rees algebras (implicitly) play in \cite{2206.13359},
we start by recalling their definition in \cref{subsection:generalized-rees-algebras},
and use them to give an alternative proof of the dictionary in dimension~1.

\subsection{Generalized Rees algebras}
\label{subsection:generalized-rees-algebras}
Generalized Rees algebras are a construction due to Van Oystaeyen \cite{MR738218}
and studied by Reiten--Van den Bergh \cite{MR0978602},
associating a type of noncommutative Rees algebra to an order;
which can be reinterpreted as a local model for the root stack.
In the rest of the subsection~$R$ denotes a noetherian integrally closed domain containing the field~$\field$
and~$\Lambda$ is a tame $R$-order:
an order for which the localisation~$\Lambda_{\mathfrak{p}}$ is hereditary for each prime~$\mathfrak{p}\triangleleft R$ of height one.

Recall that a \emph{fractional ideal} of~$\Lambda$ is
a~$\Lambda$-submodule~$I\subset \Frac(R)\Lambda$
which is reflexive as an~$R$-module and satisfies~$\Frac(R)I = \Frac(R)\Lambda$,
and that~$I$ is called \emph{divisorial} if~$I_{\mathfrak{p}}$ is invertible
for every prime~$\mathfrak{p} \in \Spec R$ of height~$1$.
The operation~$I\cdot J = (I\otimes J)^{**}$ with~$(-)^{*}$ denoting the~$R$-dual
endows the set of divisorial ideals with the structure of a group.
In particular, for any~$n\in\mathbb{Z}$ there is a well-defined symbolic power~$I^{(n)}$
given by the~$n$th power under this operation.

Reiten--Van den Bergh \cite[Chapter 5]{MR0978602} define the ramification ideal
of the order~$\Lambda$ over~$R$ as follows.
For each height-one prime~$\mathfrak{p}\triangleleft R$
the submodule~$\rad \Lambda_{\mathfrak{p}} \cap \Lambda \subset \Lambda$
is a divisorial ideal of~$\Lambda$ which contains~$\mathfrak{p}\Lambda$.
There exist a finite number of such primes~$\mathfrak{p}_1,\ldots,\mathfrak{p}_n \triangleleft R$
for which this divisorial ideal is not equal to~$\mathfrak{p} \Lambda$,
which we will denote
\begin{equation}
  \label{equation:divisorial-ideal-factors}
  P_i \colonequals \rad \Lambda_{\mathfrak{p}_i} \cap \Lambda \subset \Lambda.
\end{equation}
For each of these
there exists a minimal integer~$e_i\geq 2$
such that~$P_i^{(e_i)} \cong (\mathfrak{p}_i \Lambda)^{(1)}$,
called the \emph{ramification index} of~$\Lambda$ over~$\mathfrak{p}_i$.
\begin{definition}
  \label{definition:ramification-ideal}
  The ramification ideal of~$\Lambda$ is the product
  \begin{equation}
    \label{equation:ramification-ideal}
    \ramification(\Lambda/R) = P_1^{(e_1 -1)}\cdots P_n^{(e_n-1)}.
  \end{equation}
\end{definition}
We obtain the following construction.
\begin{definition}
  The ramification ideal defines the \emph{ramification Rees algebra}
  \begin{equation}
    \widetilde \Lambda \colonequals \bigoplus_{n\in\mathbb{Z}} \ramification(\Lambda/\kern-1pt R)^{(n)} T^n,
  \end{equation}
\end{definition}

By \cite[Theorem II.4.38]{MR1003605}
this is a maximal~$\mathbb{Z}$-graded order over its center,
and it moreover follows by the proof of~\cite[Proposition 5.1(c)]{MR0978602}
that the order $\widetilde \Lambda$ is \emph{reflexive Azumaya},
i.e.,
it is not just a hereditary order at codimension-one points,
but in fact Azumaya.
An explicit description of the center is also given in loc. cit.
as the so-called \emph{scaled Rees ring}
\begin{equation}
  \widetilde R(\mathfrak{p}_1,\ldots,\mathfrak{p}_n; e_1,\ldots, e_n) \colonequals \bigoplus_{k\in\mathbb{Z}} \mathfrak{p}_1^{\lceil k(e_1-1)/e_1 \rceil}\cdots\mathfrak{p}_n^{\lceil k(e_n-1)/e_n \rceil} T^k.
\end{equation}
Because \cite[Theorem II.4.38]{MR1003605} does not give a complete proof of this fact,
we include one for completeness under the assumption that each prime $\mathfrak{p_i}$ is a principal ideal.

\begin{lemma}
  \label{lemma:center-as-scaled-rees-ring}
  Let $\Lambda$ be a tame order over a noetherian integrally closed domain $R$,
  ramified over primes~$\mathfrak{p}_1,\ldots,\mathfrak{p}_n$.
  If each $\mathfrak{p}_i$ is principal, then $\ZZ(\widetilde\Lambda) = \widetilde R(\mathfrak{p}_1,\ldots,\mathfrak{p}_n; e_1,\ldots, e_n)$.
\end{lemma}

\begin{proof}
  We note that the center~$\ZZ(\widetilde \Lambda)$ is again graded, and its graded summands can be obtained by intersecting the graded summands of~$\widetilde \Lambda$ with~$R$ (in positive degrees) or with a suitable localisation of~$R$ in negative degrees; we consider only positive degrees for simplicity.

  For any~$k\geq 1$ the localisation of~$\ramification(\Lambda/\kern-1pt R)^{(k)}$ at a prime~$\mathfrak{p}$ is given by~$\ramification(\Lambda/\kern-1pt R)^{k}_{\mathfrak{p}} = (P_i)_{\mathfrak{p}_i}^{k(e_i-1)}$ if~$\mathfrak{p} = \mathfrak{p}_i$ for some~$i$ or~$\ramification(\Lambda/\kern-1pt R)^{k}_{\mathfrak{p}} = \Lambda_{\mathfrak{p}}$ otherwise.
  Now since~$\ramification(\Lambda/\kern-1pt R)$ is reflexive, it follows that
  \begin{equation}
    R \cap \ramification(\Lambda/\kern-1pt R)^{(k)} = R \cap \bigcap_{\mathfrak{p} \text{ height } 1} \ramification(\Lambda/\kern-1pt R)^{k}_{\mathfrak{p}} = R \cap \bigcap_{i=1}^n (R_{\mathfrak{p}_i} \cap (P_i)_{\mathfrak{p}_i}^{k(e_i-1)}).
  \end{equation}
  Letting~$k_i,r_i \in \mathbb{N}$ be integers such that~$k(e_i-1) = k_i e_i + r_i$ and~$0< r_i \leq e_i$.
  It follows from the definition of the ramification index that~$P_i^{(k(e_i-1))} = \mathfrak{p}_i^{k_i}P_i^{(r_i)}$, and so in particular
  \begin{equation}
    R_{\mathfrak{p}_i} \cap (P_i)_{\mathfrak{p}_i}^{k(e_i-1)} = \mathfrak{p}_i^{k_i}(R_{\mathfrak{p}_i} \cap (P_i)_{\mathfrak{p}_i}^{r_i}).
  \end{equation}
  Now~$R_{\mathfrak{p}_i} \cap (P_i)_{\mathfrak{p}_i}^{r_i}$ is a proper ideal of~$R_{\mathfrak{p}_i}$ which includes the maximal ideal~$R_{\mathfrak{p}_i} \cap (P_i)_{\mathfrak{p}_i}^{e_i} = \mathfrak{p}_iR_{\mathfrak{p}_i}$, hence must be equal to it.
  It follows that~$R_{\mathfrak{p}_i} \cap (P_i)_{\mathfrak{p}_i}^{k(e_i-1)} = \mathfrak{p}_i^{k_i+1}R_{\mathfrak{p}_i}$,
  and therefore
  \begin{equation}
    R \cap \ramification(\Lambda/\kern-1pt R)^{(k)}
    = R \cap \bigcap_{i=1}^n  (\mathfrak{p}_iR_{\mathfrak{p}_i})^{\lceil k(e_i-1)/e_i \rceil}
    = \mathfrak{p}_1^{\lceil k(e_1-1)/e_1 \rceil} \cap \cdots \cap \mathfrak{p}_n^{\lceil k(e_n-1)/e_n \rceil},
  \end{equation}
  where we note that $k_i + 1 = \lceil k(e_i-1)/e_i\rceil$.
  By assumption~$\mathfrak{p}_i = (a_i)$ for distinct prime elements~$a_i \in R$.
  Hence the intersection is of the claimed form
  \begin{equation*}
    \mathfrak{p}_1^{\lceil k(e_1-1)/e_1 \rceil} \cap \cdots \cap \mathfrak{p}_n^{\lceil k(e_n-1)/e_n \rceil}
    =
    (a_1^{\lceil k(e_1-1)/e_1 \rceil} \cdots a_n^{\lceil k(e_n-1)/e_n \rceil})
    =
    \mathfrak{p}_1^{\lceil k(e_1-1)/e_1 \rceil}\cdots \mathfrak{p}_n^{\lceil k(e_n-1)/e_n \rceil}.\qedhere
  \end{equation*}
\end{proof}

In what follows we will write the scaled Rees ring simply as $\widetilde R$, leaving the data implicit.

It is often convenient to work with a module-finite version of the Rees algebra,
which is constructed as follows.
Under the assumption that $\mathfrak{p}_i = (a_i)$ for some $a_i\in R$, we can take~$e = \text{lcm}({e_1,\ldots,e_n})$ to find
\begin{equation}
  \label{equation:symbolic-power-ramification}
  \ramification(\Lambda/\kern-1pt R)^{(e)} = P_1^{(e(e_1-1))}\cdots P_n^{(e(e_n-1))} = \Lambda a \Lambda,
\end{equation}
where~$a = a_1^{e(e_1-1)/e_1}\cdots a_n^{e(e_n-1)/e_n} \in R$ is central.
We then consider the~$\mathbb{Z}/e\mathbb{Z}$-graded algebra
\begin{equation}
  \label{equation:ZeZ-version}
  \Lambda_e
  \colonequals \frac{\widetilde \Lambda}{(1-aT^e)}
  = \Lambda \oplus \ramification(\Lambda/\kern-1pt R)T \oplus \cdots \oplus \ramification(\Lambda/\kern-1pt R)^{(e-1)}T^{e-1}.
\end{equation}
It follows by \cite[Proposition 5.1]{MR0978602} that~$\Lambda_e$
is a tame order over its center,
and it is again reflexive Azumaya.
Using \cref{lemma:center-as-scaled-rees-ring} the center~$R_e \colonequals \widetilde R/(1-aT^e)$, can be described explicitly as follows.

\begin{corollary}
  \label{corollary:root-stack-local}
  With notation as above, the center of~$\Lambda_e$ is given by
  \begin{equation}
    R_e = \frac{\widetilde R}{(1-aT^e)} \cong \frac{R[(a_1\cdots a_nT)^{\pm}]}{(1-aT^e)} \cong \frac{R[t]}{(t^e - a_1^{e/e_1}\cdots a_n^{e/e_n})}.
  \end{equation}
\end{corollary}

Geometrically~$\Spec R_e \to \Spec R$ is an~$e$-fold covering ramified in~$\{a_i = 0\}$ with ramification index~$e_i$.

\subsection{Hereditary orders and stacky curves}
\label{subsection:dictionary-dim-1}
The structure of hereditary orders on smooth curves is well-understood,
since they have an \'etale-local normal form described as follows.
Let~$C$ be a smooth curve and~$\order{A}$ a hereditary order on~$C$.
Given a closed point~$p$,
let~$R$ denote the henselisation of~$\mathcal{O}_{C,p}$
and~$\Lambda\colonequals \order{A}_p\otimes_C R$ the henselisation of the stalk.
Then \cite[Theorem 39.14]{MR0393100} shows that~$\Lambda$
is isomorphic to an order of the form
\begin{equation}
  \label{equation:normal-form-hereditary-order}
  \Lambda \cong \begin{pmatrix}
    R & R & \ldots & R \\
    \mathfrak{m} & R & \ldots & R \\
    \vdots & & \ddots & \vdots \\
    \mathfrak{m} & \mathfrak{m} & \ldots & R
  \end{pmatrix}^{(n_1,\ldots,n_r)}\quad \subset \Mat_n(R),
\end{equation}
where the superscript~$(n_1,\ldots,n_r)$ refers to a block decomposition of the indicated sizes,
as in \cite[Definition~39.2]{MR0393100}.
The number~$r$ is called the \emph{ramification index} of~$\Lambda$ at~$p$.
Every hereditary order is ramified at a finite number of points,
and the ramification indices of these points determine the order up to Morita equivalence.

Another object which is determined uniquely by such ramification data are \emph{smooth} stacky curves,
by the bottom-up characterization of \cite{MR3719470}.

\begin{definition}
  A \emph{stacky curve} is a separated Deligne--Mumford stack of dimension 1 with trivial generic stabilizer.
\end{definition}

In \cite[Corollary~7.8]{MR2018958} a dictionary is given which relates hereditary orders on stacky curves.

\begin{theorem}[Chan--Ingalls]
  \label{theorem:chan-ingalls}
  Let~$\order{A}$ be a hereditary order on a smooth curve~$C$.
  Then there exists
  a unique smooth stacky curve~$\stack{C}$
  whose coarse moduli space is~$C$,
  together with an equivalence
  \begin{equation}
    \label{equation:chan-ingalls-equivalence}
    \coh\stack{C}\simeq \coh(C,\order{A}).
  \end{equation}
  Conversely,
  for every smooth stacky curve~$\stack{C}$ there exists a hereditary order~$\order{A}$
  on the coarse moduli space~$C$,
  unique up to Morita equivalence,
  together with an equivalence \eqref{equation:chan-ingalls-equivalence}.
\end{theorem}
The stack~$\stack{C}$ is constructed in \cite[Theorem~7.7]{MR2018958} using an iterative procedure.

For the benefit of the reader we give a more direct construction of this dictionary using the generalized Rees algebras.
This is the 1-dimensional version of the construction in \cite{2206.13359},
and it highlights the implicit role of root stacks in \cite{MR2018958},
which predates their introduction in \cite{MR2306040}.

\paragraph{Using generalized Rees algebras}
Let~$C$ again be a smooth curve,
and~$\order{A}$ a hereditary order on $C$.
The canonical bimodule~$\omega_{\order{A}}$ is an invertible sheaf, which defines a sheaf of graded algebras
\begin{equation}
  \widetilde{\order{A}} \colonequals \bigoplus_{n\in\mathbb{Z}} \omega_{\order{A}}^{\otimes -n}.
\end{equation}
The center $\ZZ(\widetilde{\order{A}})$ is a sheaf of graded commutative algebras on $C$,
and $\widetilde{\order{A}}$ can therefore be viewed as a sheaf of graded algebras on the relative spectrum
\begin{equation}
  r'\colon \widetilde C \colonequals \relSpec_C\ZZ(\widetilde{\order{A}}) \longrightarrow C.
\end{equation}
The following lemma explains how~$\widetilde{\order{A}}$
globalizes the generalized Rees algebra construction from \cref{subsection:generalized-rees-algebras}.

\begin{lemma}
  \label{lemma:global-rees-construction}
  Over every point $p\in C$ the stalk $\widetilde{\order{A}}_p$ is the generalized Rees algebra of $\order{A}_p$.
  In particular, $\widetilde{\order{A}}$ is a graded Azumaya algebra on $\widetilde C$.
\end{lemma}

\begin{proof}
  For every~$p\in C$ the stalk $\Lambda = \order{A}_p$ is a hereditary order over $R = \mathcal{O}_{C,p}$,
  and it follows by \cite[Proposition~2.7]{MR2492474} that
  \begin{equation}
    \omega_{\order{A},p}^{-1} \cong \omega_R^{-1} \otimes_R \ramification(\Lambda/R) \cong \ramification(\Lambda/R).
  \end{equation}
  It is then immediate that $\widetilde{\order{A}}_p = \widetilde \Lambda$.
  Because $\widetilde\Lambda$ is reflexive Azumaya algebra, by definition any localisation at a height one prime is Azumaya.
  Since $R$ is a local ring of dimension $1$, this means that $\widetilde \Lambda$ itself is already Azumaya.
  Hence, so is $\widetilde{\order{A}}$.
\end{proof}

The grading on $\ZZ(\widetilde{\order{A}})$ induces an action of $\Gm$ on $\widetilde C$,
and therefore defines a quotient stack
\begin{equation}
  r\colon \stack{C}_\rt \colonequals [\widetilde{C}/\Gm] \longrightarrow C,
\end{equation}
equipped with an Azumaya algebra~$\order{A}_\rt$ corresponding to $\widetilde{\order{A}}$
under the equivalence between coherent sheaves on~$\stack{C}_\rt$
and graded coherent sheaves on $\widetilde C$.
We claim that~$\stack{C}_\rt$ is a root stack.

\begin{proposition}
  The stack~$\stack{C}_\rt$ is the root stack associated to the ramification data of $\order{A}$.
\end{proposition}

\begin{proof}
  Given any point $p\in C$ let again $\Lambda = \order{A}_p$
  denote the order over $R = \mathcal{O}_{C,p}$,
  so that $\widetilde \Lambda = \widetilde{\order{A}}_p$ as in \Cref{lemma:global-rees-construction}.
  Since $R$ is regular it follows by \cref{lemma:center-as-scaled-rees-ring} that $\ZZ(\widetilde{\order{A}})_p = \widetilde R$.
  By \cref{corollary:root-stack-local} we obtain
  \begin{equation}
    [\Spec \widetilde R/\Gm] \cong [\Spec R_e/\mu_e] = \left[\Spec \frac{R[t]}{(t^e - s)}/\mu_e\right],
  \end{equation}
  where $s\in R$ is any generator of the maximal ideal of $R$, and $e$ is the ramification index of $\order{A}_p$ at $p$.
  It follows that $\stack{C}_\rt$ is the root stack over $C$ associated to the ramification data.
\end{proof}

We aim to show that there is an equivalence as in \Cref{theorem:chan-ingalls}.

\begin{lemma}
  There exists an adjoint pair of equivalences of categories
  \begin{equation}
    \label{equation:adjunction-dim-one}
    \begin{tikzcd}
      \coh^\mathbb{Z}(\widetilde C,\widetilde{\order{A}})
      \arrow[rr, shift left, "(-)_0"]&&
      \arrow[ll, shift left, "-\otimes_{\order{A}}\widetilde{\order{A}}"]
      \coh(C,\order{A}).
    \end{tikzcd}
  \end{equation}
\end{lemma}

\begin{proof}
  It suffices to show that the adjunction between the two functors induce isomorphisms on objects.
  For each module $\cN \in \coh(C,\order{A})$ the adjunction gives rise to the natural unit map
  \begin{equation}
    \cN \to (\cN \otimes_{\order{A}} \widetilde{\order{A}})_0,\quad n \mapsto n \otimes 1.
  \end{equation}
  It is clear that this is an isomorphism, since $\widetilde{\order{A}}_0 = \order{A}$ by construction.
  Conversely, given a module $\cM \in \coh^\mathbb{Z}(\widetilde C,\widetilde{\order{A}})$,
  the adjunction gives rise to the natural counit map
  \begin{equation}
    \label{equation:adjunction-curve-global}
    \cM_0 \otimes_{\order{A}} \widetilde{\order{A}} \to \cM,
  \end{equation}
  given by the restriction of the module action on $\cM$.
  It can be checked \'etale-locally that this is an isomorphism:
  picking a closed point~$p\in C$
  and writing $R= \mathcal{O}_{C,p}^{\mathrm{sh}} \cong \field\{u\}$, $\Lambda = \order{A}_p \otimes_C R$,
  and $M= \cM_p\otimes_C R$, it suffices to consider
  \begin{equation}
    \label{equation:adjunction-curve-etale-local}
    M_0 \otimes_{\Lambda} \widetilde{\Lambda} \to M,\quad m_0 \otimes a \mapsto m_0a.
  \end{equation}
  To see this, we note that $\Lambda$ has a normal form \eqref{equation:normal-form-hereditary-order},
  and from this presentation it is easy to check that the generalised Rees algebra is of the form
  \begin{equation}
    \widetilde{\Lambda} = \bigoplus_{n\in\mathbb{Z}} \Lambda \zeta^n,
  \end{equation}
  where $\zeta \in \Lambda$ satisfies $\zeta^r = u$, and $\zeta^{-1} = \zeta^{r-1}u^{-1}$
  is its inverse in $\widetilde{\Lambda}$.
  Since every homogeneous $m_n \in M_n$ can be written as $m_n = (m_n\zeta^{-n}) \cdot \zeta^n$,
  we moreover have $M_n = M_0\zeta^n$ for each $n\in\mathbb{Z}$.
  The composition
  \begin{equation}
    M_0 \otimes_\field \field\zeta^n \xrightarrow{\ \sim\ }
    (M_0 \otimes_\Lambda \widetilde{\Lambda})_n
    \xrightarrow{\ m_0\otimes a \mapsto m_0 a\ } M_n,
  \end{equation}
  is therefore an isomorphism for every $n\in\bbZ$,
  and it follows that \eqref{equation:adjunction-curve-etale-local} is also an isomorphism.
\end{proof}

It follows by Tsen's theorem
and our standing assumption that~$\field$ is algebraically closed
that the Brauer group of a smooth (stacky) curve is trivial,
so that $\order{A}_\rt$ is necessarily a split Azumaya algebra on $\stack{C}_\rt$.
Hence, we obtain the following corollary.

\begin{corollary}
  \label{corollary:no-azumaya}
  There are equivalences $\coh(\stack{C}_\rt) \simeq \coh(\stack{C}_\rt,\order{A}_\rt) \simeq \coh(C,\order{A})$.
\end{corollary}

\subsection{Tame orders of global dimension 2 and Azumaya algebras on stacky surfaces}
\label{subsection:dictionary-dim-2}
The 1-dimensional dictionary from \cite{MR2018958}
was extended to a 2-dimensional version in \cite{2206.13359}.
We will now recall its statement and some of the details of the proof,
as they will be needed for what follows.

\begin{theorem}[Faber--Ingalls--Okawa--Satriano]
  \label{theorem:faber-ingalls-okawa-satriano}
  Let~$S$ be a normal quasiprojective surface.
  Let~$\order{A}$ be a tame order of global dimension~2.
  Then there exists a diagram
  \begin{equation}
    \label{equation:fios-diagram}
    \begin{tikzcd}
      \order{A}_\can \arrow[d, dashed, no head] & \order{A}_\rt \arrow[d, dashed, no head] & \order{A} \arrow[d, dashed, no head] \\
      \stack{S}_\can \arrow[r, "c"]    & \stack{S}_\rt \arrow[r, "r"]    & S
    \end{tikzcd}
  \end{equation}
  where
  \begin{itemize}
    \item $r\colon\stack{S}_\rt\to S$ is a root stack construction;
    \item $c\colon\stack{S}_\can\to\stack{S}_\rt$ is the canonical stack associated to~$\stack{S}_\rt$;
    \item $\order{A}_\rt$ is a reflexive Azumaya algebra on~$\stack{S}_\rt$
      such that~$r_*\order{A}_\rt\cong\order{A}$;
    \item $\order{A}_\can$ is an Azumaya algebra on~$\stack{S}_\can$
      such that~$c_*\order{A}_\can\cong\order{A}_\rt$;
  \end{itemize}
  such that there exist equivalences
  \begin{equation}
    \label{equation:fios-equivalence}
    \coh(\stack{S}_\can,\order{A}_\can)
    \simeq
    \coh(\stack{S}_\rt,\order{A}_\rt)
    \simeq
    \coh(S,\order{A}).
  \end{equation}
\end{theorem}

\paragraph{The root stack}
Consider as in \cref{theorem:faber-ingalls-okawa-satriano}
a normal connected quasi-projective surface~$S$
with a tame~$\mathcal{O}_S$-order~$\order{A}$ of global dimension~2.
Then the order~$\order{A}_\rt$ on the root stack~$r\colon \stack{S}_\rt \to S$
is defined via a global version of the generalized Rees algebra construction, see \cite[Section 3]{2206.13359}
(although the terminology does not appear as such in op.~cit.):
replacing the ramification ideal by the reflexive dual~$\omega_\order{A}^{-1}$ of the dualizing bimodule~$\omega_\order{A}$,
one obtains a sheaf of~$\mathbb{Z}$-graded algebras
\begin{equation}
  \label{equation:graded-order}
  \widetilde{\order{A}} = \bigoplus_{n\in\mathbb{Z}} \omega_{\order{A}}^{(n)},
\end{equation}
which we can interpret as a $\mathbb{Z}$-graded reflexive Azumaya algebra over the relative affine scheme
\begin{equation}
  \label{equation:root-stack-S}
  r'\colon \widetilde S \colonequals \relSpec_S \ZZ(\widetilde{\order{A}}) \longrightarrow S.
\end{equation}
The grading on the center yields a~$\Gm$-action on~$\widetilde S$.
The map $r'$ is invariant for the action and therefore factors over the quotient to a map
\begin{equation}
  r\colon \stack{S}_\rt = [\widetilde S/\Gm] \to S.
\end{equation}
The order~$\order{A}_\rt$ of \cite[\S3]{2206.13359}
is the sheaf of algebras corresponding to the sheafification~$\widetilde{\order{A}}$
along the equivalence~$\coh \stack{S}_\rt \simeq \coh^{\mathbb{Z}} \widetilde S$.
This order satisfies~$r_*\order{A}_\rt \cong \order{A}$ by \cite[Lemma 4.3]{2206.13359}
and furthermore there are equivalences of categories
\begin{equation}
  \begin{tikzcd}[column sep=large]
    \coh(\stack{S}_\rt,\order{A}_\rt)\arrow[rr, shift left, "r_*"] & & \coh(S,\order{A}) \arrow[ll, shift left, "r^*(-)\otimes_{r^*\order{A}}\order{A}_\rt"].
  \end{tikzcd}
\end{equation}
We remark that~$r$ is an isomorphism over~$S\setminus \Delta$,
where~$\Delta$ denotes the ramification locus of~$\order{A}_\rt$,
and that~$\order{A}_\rt$ is isomorphic to~$\order{A}$ when restricted to~$S\setminus \Delta$.

It will be useful for the rest of the paper to detail the relation between \cite{2206.13359} and \cite{MR0978602}.
Since~$S$ is normal, every point in~$S$ admits a neighbourhood~$U\subset S$ such that~$R = \Gamma(U,\mathcal{O}_S)$
is an integrally closed domain,
equipped with a tame order~$\Lambda = \Gamma(U,\order{A})$.
Writing~$P_1,\ldots,P_n \subset \Lambda$ for the ramified primes
over~$\mathfrak{p}_1,\ldots,\mathfrak{p}_n \triangleleft R$ as before,
the formula in \cite[Proposition~2.7]{MR2492474}
shows that the dualising bimodule of~$\Lambda$ is given by
\begin{equation}
  \omega_\Lambda = \omega_R \otimes_R \left( \prod_{i=1}^nP_i\mathfrak{p}_i^{-1} \right).
\end{equation}
Since~$S$ is normal, we may shrink~$U$ such that~$\omega_R^{**}$ is trivial.
It then follows that the reflexive dual~$\omega_\Lambda^{-1}$
is given by~$\ramification(\Lambda/\kern-1pt R)$ from \eqref{equation:ramification-ideal} as in \cite[Remark~2.12]{2206.13359}.
In particular, there is an isomorphism of~$\mathbb{Z}$-graded algebras
\begin{equation}
  \Gamma(U,\widetilde{\order{A}}) \cong \widetilde \Lambda,
\end{equation}
and the root stack is locally given by the map~$\widetilde S \times_S U \cong \Spec \widetilde R \to \Spec R$.
If the primes~$\mathfrak{p}_1,\ldots,\mathfrak{p}_n$ are principal, which can be assumed if for example~$S$ is locally factorial, then one can again define the finite quotients~$\Lambda_e$ and~$R_e$.
It follows from \cite[Corollary 2.15]{2206.13359} that there is an isomorphism of stacks
\begin{equation}
  \stack{S}_\rt \times_S U \cong [\Spec \widetilde R/ \Gm] \cong [\Spec R_e/\mu_e],
\end{equation}
which identifies the sheaves of orders associated to~$\widetilde \Lambda$ and~$\Lambda_e$
with~$\order{A}_\rt|_{\stack{S}_\rt\times_S U}$.
In particular, the local structure of~$\stack{S}_\rt$ is described by \cref{corollary:root-stack-local}.

\begin{remark}
  By inspecting the construction of the root stack
  using generalized Rees algebras
  in \cite{2206.13359} for orders on surfaces,
  as just recalled,
  or \cref{subsection:dictionary-dim-1} for orders on curves,
  and the papers upon which this construction builds in the affine case,
  namely \cite[\S5]{MR0978602}
  (see also \cite[\S2]{MR2492474})
  one notices that the construction should also work
  in a higher-dimensional setting.
  It would then give rise to an equivalence between
  \begin{itemize}
    \item the category of left modules over a tame order of finite global dimension over a smooth variety;
    \item the category of left modules over a reflexive Azumaya algebra (which is Azumaya if the ramification divisor is smooth)
      on the associated root stack.
  \end{itemize}
  We do not address this generalization,
  as we are only concerned with 1- and 2-dimensional objects.
\end{remark}

\paragraph{The canonical stack}
It is shown in \cite{2206.13359} that~$\stack{S}_\rt$ is
a Deligne--Mumford stack
with linearly reductive quotient singularities.
Such a stack has a natural resolution given by the canonical stack~$c\colon \stack{S}_\can \to \stack{S}_\rt$ of \cite{MR1005008}.
Let us recall the complete local description of the canonical stack given in \cite[\S4]{2206.13359}.
Around each point~$p\in S$ we can consider a complete local chart~$\Spec R \subset S$
such that~$\stack{S}_\rt \to S$ can be locally presented as
\begin{equation}
  [\Spec R_e/\mu_n] \to \Spec R,
\end{equation}
where~$R_e = R[t]/(t^e-a_1^{e/e_1}\cdots a_n^{e/e_n})$ as in \Cref{corollary:root-stack-local}.
By construction $R_e$ can be presented as an invariant ring~$R_e \cong T^H \subset T$
of a finite cyclic group~$H$ acting on~$T = k[\![x,y]\!]$.
The canonical stack is constructed such that~$c\colon \stack{Y}_\can \to \stack{Y}_\rt$
pulls back along the \'etale cover $\Spec R_e \to [\Spec R_e/\mu_n]$ to the map
\begin{equation}
  [\Spec T/H] \to \Spec R_e.
\end{equation}
This local description shows that this is a birational modification in codimension~$2$.
In \cite{2206.13359} the authors construct the~$\mathcal{O}_{\stack{S}_\can}$-order~$\order{A}_\can$ as the reflexive hull
\begin{equation}
  \order{A}_\can \colonequals (c^*\order{A}_\rt)^{\vee\vee},
\end{equation}
of the pullback of~$\order{A}_\rt$,
and show that this is a sheaf of Azumaya algebras satisfying~$c_*\order{A}_\can \cong \order{A}_\rt$.

The order can be locally described as follows.
Writing~$\Lambda_e$ for the pullback of~$\order{A}_\rt$ to~$\Spec R_e$,
the pullback of~$\order{A}_\can$ to~$\Spec T$ is given by the reflexive closure
\begin{equation}
  \Gamma = (\Lambda_e \otimes_{R_e} T)^{**}.
\end{equation}
There is an induced action of $H$ on $\Gamma$ and it is shown in \cite{2206.13359} that there is an equivalence
\begin{equation}\label{eq:cancomlocequiv}
  \begin{tikzcd}[column sep=huge]
    \mmod^H \Gamma
    \arrow[rr, shift left, "{\Hom_\Gamma^H(\Gamma,-)}"]&&
    \arrow[ll, shift left, "{-\otimes_{\Lambda_e} \Gamma}"]
    \overset{\sim}\longleftrightarrow \mmod \Lambda_e.
  \end{tikzcd}
\end{equation}
It is shown in \cite{2206.13359} that this induces a global equivalence
\begin{equation}\label{eq:cancatequiv}
\begin{tikzcd}[column sep=huge]
  \coh(\stack{S}_\can,\order{A}_\can)\arrow[rr, shift left, "c_*"]&& \coh(\stack{S}_\rt,\order{A}_\rt) \arrow[ll, shift left, "c^*(-)\otimes_{c^*\order{A}_\rt}\order{A}_\can"].
\end{tikzcd}
\end{equation}

We end this section by explicitly spelling out some properties of the bottom-up construction.

\begin{proposition}
  \label{proposition:r-c-isomorphism}
  Let~$p \in S$ and consider the scheme~$S_p = \Spec \mathcal{O}_{S,p} \subset S$
  associated to the local ring~$\mathcal{O}_{S,p}$ at~$p$. Then:
  \begin{enumerate}
    \item\label{enumerate:r-c-isomorphism-1}
      $p\not\in \Delta$
      if and only if~$\stack{S}_\rt \times_S S_p \to S_p$ is an isomorphism,
    \item\label{enumerate:r-c-isomorphism-2}
      $p \not\in (\Delta^\sing \cup S^\sing)$
      if and only if~$\stack{S}_\can \times_S S_p \to \stack{S}_\rt \times_S S_p$ is an isomorphism.
  \end{enumerate}
\end{proposition}

\begin{proof}
  \cref{enumerate:r-c-isomorphism-1}
  If~$p \not\in \Delta$ then it follows that the algebra~$A = \order{A} \otimes_S \mathcal{O}_{S,p}$
  is unramified on~$R = \Gamma(S_p,\mathcal{O}_{S,p})$, hence~$\omega_A^{-1} = A$ is trivial.
  The Rees algebra is therefore of the form~$\widetilde A = \bigoplus_{k\in\mathbb{Z}} A$,
  and its center is simply~$\widetilde R = R[T^{\pm}]$.
  Therefore the~$\Gm$-action on~$\Spec \widetilde R = \Spec R \times \Gm$ is free and induces an isomorphism
  \begin{equation}
    \stack{S}_\rt\times_S S_p = [\Spec \widetilde R/\Gm] \cong \Spec R = S_p.
  \end{equation}

  \cref{enumerate:r-c-isomorphism-2}
  The map~$\stack{S}_\can\to \stack{S}_\rt$ is by construction
  an isomorphism outside of the singularities of~$\stack{S}_\rt$.
  Hence it suffices to determine to prove that~$\stack{S}_\rt \times_S S_p$ is smooth
  if and only if~$p$ does not hit the singular loci of~$S$ and~$\Delta \subset S$.
  If~$p \in S$ is smooth with~$p \not\in \Delta$,
  then it follows from \cref{enumerate:r-c-isomorphism-1} that~$\stack{S}_\rt\times_S S_p \cong S_p$ is again smooth,
  and conversely,
  so we are done.
  Hence we can suppose that~$p\in S$ is a smooth point with~$p\in \Delta$ and~$p \not\in\Delta^\sing$.

  Writing again~$A = \order{A} \otimes_S \mathcal{O}_{S,p}$
  for the order on~$R = \Gamma(S_p,\mathcal{O}_{S,p})$,
  we note that the condition~$p \not\in\Delta^\sing$
  is equivalent to the fact that~$A$
  is ramified at exactly one prime~$\mathfrak{p}_1$ with some ramification index~$e \in \mathbb{N}$.
  Since~$p$ is smooth,~$R$ is a UFD and hence~$\mathfrak{p}_1 = (a_1)$ for some irreducible element~$a_1$.
  It follows from \cref{corollary:root-stack-local} that
  \begin{equation}
    \stack{S}_\rt \times_S S_p = \left[\Spec R_e/\mu_e\right] \cong \left[\Spec \frac{R[t]}{(t^e - a_1)}/ \mu_e\right].
  \end{equation}
  Now because~$p \not\in\Delta^\sing$ it follows that~$R/(a_1)$ is smooth,
  and hence~$R[t]/(t^e-a_1)$ is also smooth (for example by the Jacobian criterion for smoothness).
  It follows that~$\stack{S}_\rt \times_S S_p$ is smooth
  and hence~$\stack{S}_\can\times_S S_p \to \stack{S}_\rt\times_S S_p$ is an isomorphism.
  This also proves the converse.
\end{proof}

\section{Restricting to central curves}
\label{section:restriction}
We now come to the first main result of the paper,
which addresses what happens when
we take an order~$\order{A}$ on a surface~$S$,
together with a curve~$C\subset S$,
and consider~$(C,\order{A}|_C)$.
This will usually be an order on~$C$,
and we will explain how the dictionary from \cref{subsection:dictionary-dim-2}
can be used to understand~$(C,\order{A}|_C)$.

\subsection{Restricting the equivalence}
We work in the setting of \cref{subsection:dictionary-dim-2},
with~$\order{A}$ a tame order of global dimension~2,
where we in addition assume that~$S$ is a \emph{smooth} quasiprojective surface.
Let $C\subset S$ be a curve.
We will consider the diagram of fiber products
\begin{equation}
  \label{equation:restriction-diagram}
  \begin{tikzcd}
    \stack{C}_\can \arrow[r, "c_C"] \arrow[d, hookrightarrow] \arrow[dr, phantom, "\usebox\pullback", very near start, color=black]
    & \stack{C}_\rt \ar[r, "r_C"] \arrow[d, hookrightarrow] \arrow[dr, phantom, "\usebox\pullback", very near start, color=black]
    & C \ar[d,hookrightarrow] \\
    \stack{S}_\can \ar[r, "c"]
    & \stack{S}_\rt \ar[r, "r"]
    & S
  \end{tikzcd}
\end{equation}
induced from the bottom row in \eqref{equation:fios-diagram}.
Restricting the orders on the (stacky) surfaces to each of these (stacky) curves yields the sheaves of algebras
\begin{equation}
  \label{equation:notation-restricted-sheaves-of-algebras}
  \order{B} \colonequals \order{A}|_C,\quad
  \order{B}_\rt \colonequals \order{A}_\rt|_{\stack{C}_\rt},\quad
  \order{B}_\can \colonequals \order{A}_\can|_{\stack{C}_\can}.
\end{equation}
on~$C$, $\stack{C}_\rt$ and~$\stack{C}_\can$ respectively.
We will show that these sheaves of algebras are related via the maps $c_C$ and $r_C$,
and that the associated categories of modules are again equivalent.
Namely,
we have the following precise incarnation of \cref{theorem:restriction-introduction}.
\begin{theorem}
  \label{theorem:restriction}
  Let~$S$ be a smooth quasiprojective surface,
  and let~$\order{A}$ be a tame order of global dimension~2.
  Let~$C\subset S$ be a curve.
  With the above notation,
  there exist isomorphisms~$c_{C,*}\order{B}_\can\cong\order{B}_\rt$
  and~$r_{C,*}\order{B}_\rt\cong\order{B}$,
  and equivalences of categories
  \begin{equation}
    \label{equation:restriction-equivalences}
    \coh(C,\order{B})
    \simeq \coh(\stack{C}_\rt,\order{B}_\rt)
    \simeq \coh(\stack{C}_\can,\order{B}_\can).
  \end{equation}
\end{theorem}
The precise shape of the functors is given in the proof.

We first consider the restriction of the root stack construction.
\begin{proposition}
  \label{proposition:root-stack-restriction}
  There is an isomorphism~$\order{B} \cong r_{C,*}\order{B}_\rt$ which induces an equivalence
  \begin{equation}
    \coh(C,\order{B}) \simeq \coh(\stack{C}_\rt,\order{B}_\rt).
  \end{equation}
\end{proposition}

\begin{proof}
  Recall that the root stack~$\stack{S}_\rt$ is defined as the~$\Gm$-quotient of the affine map~$r'$ in \eqref{equation:root-stack-S},
  defined by the~$\mathbb{Z}$-graded~$\mathcal{O}_S$-algebra~$\mathcal{O}_{\widetilde{S}} = \ZZ(\widetilde{\order{A}})$.
  Writing~$\mathcal{I}\triangleleft\mathcal{O}_S$ for the sheaf of ideals associated to $C$,
  the base change of~$r'$ along~$C \hookrightarrow S$ is again an affine map
  \begin{equation}
    r'_C\colon \widetilde{C} = \relSpec_C \mathcal{O}_{\widetilde{S}}/\mathcal{I}\mathcal{O}_{\widetilde{S}} \to C,
  \end{equation}
  and the fibre product~$\stack{S}_\rt\times_S C$ is given by the quotient~$r_C\colon \stack{C}_\rt = [\widetilde{C}/\Gm] \to C$.
  The pushforward along~$r_{C,*}$ is given by the composition
  \begin{equation}
    \coh^{\Gm}(\widetilde C)
    \xrightarrow[\sim]{\ r'_{C,*}\ }
    \coh^{\mathbb{Z}}(C,\mathcal{O}_{\widetilde{S}}/\mathcal{I}\mathcal{O}_{\widetilde{S}})
    \xrightarrow{(-)_0}
    \coh^{\mathbb{Z}}(C)
  \end{equation}
  sending an equivariant sheaf on~$\widetilde C$ to the degree-$0$ part of the corresponding~$\cO_{\widetilde S}/\mathcal{I}\cO_{\widetilde S}$-module on $C$.
  Along this pushforward, the restricted algebra~$\order{B}_\rt \in \coh \stack{C}_\rt = \coh^\Gm(\widetilde C)$ maps to
  \begin{equation}
    r_{C,*}\order{B}_\rt
    = (r'_{C,*}\order{B}_\rt)_0
    = (\widetilde{\order{A}}/\mathcal{I}\widetilde{\order{A}})_0
    = \widetilde{\order{A}}_0/\mathcal{I} \widetilde{\order{A}}_0
    \cong \order{A}/\mathcal{I}\order{A}
    = \order{B}.
  \end{equation}
  Hence, if we denote~$\widetilde{\order{B}} = \widetilde{\order{A}}/\mathcal{I}\widetilde{\order{A}}$
  we obtain two adjoint functors as in \cite[Lemma 4.4]{2206.13359}:
  \begin{equation}
    \begin{tikzcd}[column sep=3cm]
      \coh(C,\order{B}) \arrow[r,shift left, "G = (-)\otimes_{\order{B}} \widetilde{\order{B}}"]
      & \arrow[l,shift left,"F= (-)_0"]
      \coh^{\mathbb{Z}}(C,\widetilde{\order{B}})
      \simeq \coh^{\Gm}(\stack{C}_\rt,\order{B}_\rt).
    \end{tikzcd}
  \end{equation}
  Similar to \cite[Lemma 4.4]{2206.13359}, we already know that for a module~$\mathcal{M} \in \coh(C,\order{B})$
  the natural map~$\mathcal{M} \to F(G(\mathcal{M}))$ is an isomorphism.
  Conversely, identifying an object~$\mathcal{N} \in \coh^{\mathbb{Z}}(C,\widetilde{\order{B}})$
  with the corresponding~$\mathcal{I}$-torsion object
  of~$\coh^{\mathbb{Z}}(S,\widetilde{\order{A}})$,
  the natural map~$G(F(\mathcal{N})) \to \mathcal{N}$ fits into a commutative diagram
  \begin{equation}
    \begin{tikzcd}
      G(F(\mathcal{N})) = \mathcal{N}_0 \otimes_{\order{B}} \widetilde{\order{B}} \arrow[r] \arrow[d, "\sim"]
      & \mathcal{N} \arrow[d, "\sim"] \\
      (\mathcal{N} \otimes_{\widetilde{\order{A}}} \widetilde{\order{B}})_0 \otimes_{\order{A}} \widetilde{\order{A}} \arrow[r]
      & \mathcal{N} \otimes_{\widetilde{\order{A}}} \widetilde{\order{B}}
    \end{tikzcd}
  \end{equation}
  It is shown in \cite[Lemma 4.4]{2206.13359} that the bottom map is an isomorphism, hence so is the top.
\end{proof}

For the canonical stack we first consider the complete local description in \Cref{subsection:dictionary-dim-2}
where~$\Gamma$ is an~$H$-equivariant order on~$\Spec T$
for~$T=\field[\![x,y]\!]$
obtained from an order~$\Lambda_e$
on a complete local chart~$\Spec R_e$ of the root stack.
We have the following.
\begin{lemma}
  \label{lemma:canonical-local-slice}
  If $J\subset R_e = T^H$ is an ideal,
  then \eqref{eq:cancomlocequiv} induces an equivalence
  \begin{equation}
    \begin{tikzcd}[column sep=4cm]
      \mmod^H \Gamma/J\Gamma
      \arrow[r, shift left, "{\Hom_{\Gamma\!/\!J}^H(\Gamma\!/\!J\Gamma,-)}"]
      & \mmod \Lambda_e/J\Lambda_e.
      \arrow[l, shift left, "{-\otimes_{\Lambda_e\!/\!J\Lambda_e} \Gamma\!/\!J\Gamma}"]
    \end{tikzcd}
  \end{equation}
\end{lemma}

\begin{proof}
  Using \cite[Proposition 2.26]{2206.13359}
  the inclusion~$\Lambda_e \to (\Lambda_e \otimes_{R_e} T)^{**} = \Gamma$
  induces an isomorphism~$\Lambda_e \cong \Gamma^H$,
  and via this identification there are isomorphisms
  \begin{equation}
    \Gamma * H \cong \End_{\Gamma^H}(\Gamma) \cong \End_{\Lambda_e}(\Gamma).
  \end{equation}
  These isomorphisms are~$R_e$-linear along the isomorphism~$R_e \cong T^H$,
  and therefore they induce $R_e$-linear equivalences
  \begin{equation}
    \begin{tikzcd}[column sep=huge]
      \mmod^H \Gamma
      \arrow[r,"{\Hom_\Gamma^H(\Gamma,-)}"]&
      \mmod \End_{\Lambda_e}(\Gamma)
      \arrow[r,"\sim"]&
      \mmod \Lambda_e,
    \end{tikzcd}
  \end{equation}
  where the second equivalence is the Morita equivalence,
  using the fact that~$\Gamma$ is a projective~$\Lambda_e$-module which contains~$\Lambda_e$ as a direct summand.
  It follows from~$R_e$-linearity that these equivalences restrict to the subcategories of $J$-torsion objects
  \begin{equation}
    \mmod^H \Gamma/J\Gamma \simeq \{ M \in \mmod^H \Gamma \mid JM = 0\},\quad
    \mmod \Lambda_e/J\Lambda_e \simeq \{ M \in \mmod \Lambda_e \mid JM = 0\}
  \end{equation}
  for any ideal~$J\triangleleft R_e \cong T^H$.
\end{proof}

This allows us to discuss the restriction of the canonical stack construction.

\begin{proposition}
  \label{proposition:canonical-stack-restriction}
  There is an isomorphism~$\order{B}_\rt \cong c_{C,*}\order{B}_\can$
  which induces an equivalence
  \begin{equation}
    \coh(\stack{C}_\rt,\order{B}_\rt)\simeq \coh(\stack{C}_\can,\order{B}_\can)
  \end{equation}
\end{proposition}

\begin{proof}
  By the construction in \eqref{equation:restriction-diagram},
  the stacky curve in the canonical stack is given by the fiber product
  \begin{equation}
    \begin{tikzcd}
      \stack{C}_\can \arrow[r, "i"] \arrow[d, "c_C"] & \stack{S}_\can \ar[d, "c"] \\
      \stack{C}_\rt \arrow[r, "j"] & \stack{S}_\rt
    \end{tikzcd}
  \end{equation}
  The usual adjunctions yield a canonical base-change map~$j^*c_*\mathcal{F} \to c_{C,*}i^*\mathcal{F}$
  for any sheaf~$\mathcal{F}$ on~$\stack{S}_\can$.
  In particular, since~$c_*\order{A}_\can = \order{A}_\rt$ by \cref{theorem:faber-ingalls-okawa-satriano},
  we find a canonical morphism
  \begin{equation}
    \label{equation:base-change-morphismh}
    \order{B}_\rt = j^*c_*\order{A}_\can \to c_{C,*}i^*\order{A}_\can = c_{C,*}\order{B}_\can.
  \end{equation}
  To check that this is an isomorphism, it suffices to check it is an isomorphism
  on complete local neighbourhoods of points~$p\in C \subset S$ in~$S$.
  For any such point we can consider the root stack over a complete local chart
  \begin{equation}
    [\Spec R_e/\mu_e] \to \Spec R,
  \end{equation}
  and let~$J\subset R$ be the ideal cutting out~$C$.
  Then the stacky curve~$\mathcal{C}_\rt$ restricts to the quotient stack
  \begin{equation}
    \left[\Spec R_e / \mu_e\right] \times_{\Spec R} \Spec(R/JR) \cong \left[\Spec(R_e/J R_e) / \mu_e\right].
  \end{equation}
  Base changing along the \'etale cover~$\Spec R_e$,
  the map~$c_C\colon\stack{C}_\can\to\stack{C}_\rt$
  restricts to
  \begin{equation}
    \left[\Spec(T/J T) /H\right] \to \Spec (R_e/JR_e).
  \end{equation}
  In particular, the sheaves of algebras~$\order{B}_\can = \order{A}_\can|_{\stack{C}_\can}$
  and~$\order{B}_\rt = \order{A}_\rt|_{\stack{C}_\rt}$
  restrict on~$\Spec T/JT$ resp.~$\Spec R_e/JR_e$
  to the quotients~$\Gamma/J\Gamma$ and~$\Lambda_e/J\Lambda_e$.
  To describe the map \eqref{equation:base-change-morphismh},
  first note that~$J$ is a principal ideal~$J=(a)$ since~$C$ is a divisor on a smooth surface~$S$.
  We therefore find a short exact sequence
  \begin{equation}
    \label{equation:short-exact-sequence-Gamma-J}
    0 \to \Gamma \xrightarrow{a\cdot} \Gamma \to \Gamma/J\Gamma \to 0,
  \end{equation}
  where we note that multiplication by~$a$ is injective
  since~$\Gamma$ is torsion-free and~$T$ is a domain.
  On this local chart the direct image is given
  by the functor~$\Hom_{T*H}(T,-) \colon \mmod T*H \to \mmod T^H = \mmod R_e$,
  and applying it to \eqref{equation:short-exact-sequence-Gamma-J} yields the long exact sequence
  \begin{equation}
    0 \to \Hom_{T*H}(T,\Gamma) \xrightarrow{a\cdot} \Hom_{T*H}(T,\Gamma) \to \Hom_{T*H}(T,\Gamma/J\Gamma) \to \Ext^1_{T*H}(T,\Gamma) \to \ldots
  \end{equation}
  Because~$H$ is a finite group,~$T$ is projective as a~$T*H$-module
  and therefore~$\Ext^1_{T*H}(T,\Gamma) = 0$.
  By \cite[Lemma 4.3]{2206.13359} there is moreover
  a canonical isomorphism~$\Hom_{T*H}(T,\Gamma) \cong \Lambda_e$
  and we therefore obtain an exact sequence
  \begin{equation}
    0\to \Lambda_e \xrightarrow{a\cdot} \Lambda_e \to \Hom_{T*H}(T,\Gamma/J\Gamma) \to 0,
  \end{equation}
  which implies that~$\Hom_T^H(T,\Gamma/J\Gamma)\cong \Lambda_e/J\Lambda_e$.
  We conclude that~$\order{B}_\rt \cong c_{C,*}\order{B}_\can$.

  For the second statement,
  using \cref{lemma:embedding-as-torsion}
  the full subcategories~$\coh(\stack{C}_\rt,\order{B}_\rt) \subset \coh(\mathcal{S}_\rt,\mathcal{A}_\rt)$
  and~$\coh(\stack{C}_\can,\order{B}_\can) \subset \coh(\mathcal{S}_\can,\mathcal{A}_\can)$
  are identified
  with the subcategories of~$r^*\mathcal{I}$-torsion and~$c^*r^*\mathcal{I}$-torsion objects,
  where~$\mathcal{I}$ is the defining sheaf of ideals of~$C$ in~$S$.
  It then suffices to show that the inverse pair of functors
  \begin{equation}
    G\colon \coh(\stack{S}_\rt,\order{A}_\rt) \leftrightarrow \coh(\stack{S}_\can,\order{A}_\can) \ :\! F,
  \end{equation}
  preserve these subcategories of torsion objects.
  It again suffices to check this complete locally,
  and the result therefore follows from \Cref{lemma:canonical-local-slice}.
\end{proof}

\begin{proof}[Proof of \cref{theorem:restriction}]
  This follows by combining \cref{proposition:root-stack-restriction} and \cref{proposition:canonical-stack-restriction}.
\end{proof}

\subsection{Properties of the restriction}
\label{subsection:properties-restriction}
We will now study the properties of the sheaf of algebras~$\order{A}|_C$.
We are mostly interested in the case where~$\order{A}|_C$ is an order,
as the geometry of orders is well-understood.
The following is
the more precise version of \cref{theorem:restriction-properties-introduction}\cref{enumerate:restriction-order}.
\begin{proposition}
  \label{proposition:restriction-order}
  Let~$S,\mathcal{A},C$ be as in \cref{theorem:restriction},
  and assume moreover that~$C$ is integral.
  The sheaf of algebras~$\order{A}|_C$ is an order if and only if~$C$ is not contained in~$\Delta$.
  Moreover, when~$\order{A}|_C$ is an order,
  \begin{itemize}
    \item the restriction of~$\order{A}|_C$ to the dense open Azumaya locus is split Azumaya,
    \item the discriminant of~$\order{A}|_C$ is~$C\cap\Delta$.
  \end{itemize}
\end{proposition}

\begin{proof}
  Let~$U\colonequals S\setminus\Delta$ be the Azumaya locus of~$\order{A}$.
  If~$C$ is not contained in~$\Delta$,
  the restriction~$\order{A}|_C$ is Azumaya on the non-empty Zariski-open subset~$C\cap U\subseteq C$,
  so that~$\order{A}|_C$ is an order in~$\order{A}\otimes_S\field(C)$.

  Conversely,
  if~$C$ is contained in~$\Delta$,
  then by the fiberwise characterization of Azumaya algebras \cite[Proposition~IV.2.1(b)]{MR559531},
  the fibers of~$\order{A}$ at~$p\in C$ are not central simple algebras
  (because otherwise~$C\cap U\neq\emptyset$),
  so~$\order{A}|_C$ cannot be Azumaya on some dense open.

  For the first point in the moreover part, it suffices to observe that
  the Brauer group of \emph{any} curve over an algebraically closed field is trivial,
  by \cite[Proposition~7.3.2]{MR4304038}.
  The second point again follows from the fiberwise characterization \cite[Proposition~IV.2.1(b)]{MR559531},
\end{proof}

The following proposition is a more precise version
of \cref{theorem:restriction-properties-introduction}\cref{enumerate:restriction-azumaya}.
\begin{proposition}
  \label{proposition:restriction-azumaya}
  Let~$S,\mathcal{A},C$ be as in \cref{theorem:restriction},
  and assume moreover that~$C$ is integral.
  Assume that~$\order{A}|_C$ is an order.
  Then the following are equivalent:
  \begin{enumerate}
    \item\label{enumerate:restriction-azumaya-azumaya} $\order{A}|_C$ is an Azumaya algebra;
    \item\label{enumerate:restriction-azumaya-discriminant} $C\cap\Delta=\emptyset$.
    \item\label{enumerate:restriction-azumaya-isomorphism} the morphism~$r_C\circ c_C\colon\stack{C}_\can\to C$ is an isomorphism;
  \end{enumerate}
\end{proposition}

\begin{proof}
  The equivalence $\cref{enumerate:restriction-azumaya-azumaya}\Leftrightarrow\cref{enumerate:restriction-azumaya-discriminant}$
  follows from the characterization of Azumaya algebras in the fibers \cite[Proposition~IV.2.1(b)]{MR559531},
  because~$\mathcal{A}$ is already locally free,
  as~$S$ is assumed to be smooth.

  The equivalence $\cref{enumerate:restriction-azumaya-discriminant}\Leftrightarrow\cref{enumerate:restriction-azumaya-isomorphism}$
  follows from \cref{proposition:r-c-isomorphism}:
  the maps~$r$ and~$c$ are an isomorphism if and only if~$C\subset S\setminus\Delta$.
\end{proof}

Finally,
the following proposition is a more precise version
of \cref{theorem:restriction-properties-introduction}\cref{enumerate:restriction-hereditary}.
\begin{proposition}
  \label{proposition:restriction-hereditary}
  Let~$S,\mathcal{A},C$ be as in \cref{theorem:restriction},
  and assume moreover that~$C$ is integral.
  We will use the notation introduced in \eqref{equation:notation-restricted-sheaves-of-algebras}.

  If~$C$ is smooth and intersects~$\Delta$ transversely in the smooth locus of~$\Delta$,
  then
  \begin{enumerate}
    \item
      \label{enumerate:restriction-hereditary-root-stack}
      The morphism~$c_C\colon\stack{C}_\can\to\stack{C}_\rt$ is an isomorphism,
      and the stacks are isomorphic to the root stack in the reduced divisor~$C\cap\Delta$
      with the multiplicities given by the ramification indices of the order~$\order{A}|_C$.
    \item
      \label{enumerate:restriction-hereditary-hereditary}
      The order~$\order{B}$ is hereditary.
  \end{enumerate}
  If~$C$ intersects~$\Delta$ non-transversely in the smooth locus of~$\Delta$,
  then
  \begin{enumerate}
      \setcounter{enumi}{2}
    \item The morphism~$c_C\colon\stack{C}_\can\to\stack{C}_\rt$ is an isomorphism.
      They are isomorphic to the root stack in the non-reduced divisor~$C\cap\Delta$
      with the multiplicities given by the ramification indices.
    \item The order~$\order{B}$ is not hereditary.
  \end{enumerate}
\end{proposition}

\begin{proof}
  In both cases the isomorphism of stacks follows from \cref{proposition:r-c-isomorphism}.
  The description of the root stack follows from the functoriality properties of the root stack construction.

  Because global dimension is an invariant of the abelian categories,
  the equivalence \eqref{equation:restriction-equivalences}
  explains when~$\order{B}$ is (not) hereditary.
  Namely,
  the root stack~$\stack{C}$ is smooth if and only if
  the divisor we consider is reduced,
  which is determined by the (non-)transversality of the intersection~$C\cap\Delta$.
\end{proof}

If~$\order{A}$ is terminal and~$C$ is smooth,
we can upgrade \cref{proposition:restriction-hereditary} to the following.
It would be interesting to further extend this to all tame orders and integral curves.
\begin{proposition}
  \label{proposition:restriction-terminal-order}
  Let $S,\order{A}, C$ be as in \cref{theorem:restriction}
  and assume moreover that $C$ is smooth and $\order{A}$ is terminal.
  Then the following are equivalent.
  \begin{enumerate}
    \item \label{enumerate:restriction-hereditary-root-stack-terminal}
      The morphism~$c_C\colon\stack{C}_\can\to\stack{C}_\rt$ is an isomorphism,
      and the stacks are isomorphic to the root stack in the reduced divisor~$C\cap\Delta$
      with the multiplicities given by the ramification indices of the order~$\order{A}|_C$.
    \item \label{enumerate:restriction-hereditary-hereditary-terminal}
      The order~$\order{A}|_C$ is hereditary.
    \item \label{enumerate:intersection-transversal-hereditary}
      The curve $C$ intersects $\Delta$ transversely in the smooth locus of~$\Delta$.
  \end{enumerate}
\end{proposition}

\begin{proof}
  The equivalence of \cref{enumerate:restriction-hereditary-root-stack-terminal}
  and \cref{enumerate:restriction-hereditary-hereditary-terminal}
  follows from \cref{theorem:chan-ingalls} and \cref{proposition:r-c-isomorphism}.
  That \cref{enumerate:intersection-transversal-hereditary}
  implies \cref{enumerate:restriction-hereditary-root-stack-terminal},
  and that \cref{enumerate:intersection-transversal-hereditary}
  implies \cref{enumerate:restriction-hereditary-hereditary-terminal}
  follows from \cref{proposition:restriction-hereditary}.

  For the remaining implication from \cref{enumerate:restriction-hereditary-hereditary-terminal}
  to \cref{enumerate:intersection-transversal-hereditary},
  we know already from \cref{proposition:restriction-hereditary}
  that $\order{A}\vert_C$ is not smooth if $C$ intersects $\Delta$ non-transversely in the smooth locus.

  Hence, assume that $p\in C\cap \Delta$ lies in the singular locus of $\Delta$.
  We are going to show that the fiber $\order{A}(p) = \order{A}_p\otimes_{\mathcal{O}_{S,p}}\mathcal{O}_{S,p}/\mathfrak{m}_{S,p}$
  is not the fiber of a hereditary order over a curve.
  For this we can use the \'etale local description of terminal orders
  (the standard form being given in \cite[Definition 2.6]{MR2180454})
  and describe the fiber as a path algebra with relations isomorphic to the finite-dimensional algebra $\order{A}(p)$.
  From \cref{lemma:finite-dimensional-algebra-from-terminal-order} and \cref{lemma:quiver-of-terminal-order} below,
  it follows that this fiber cannot be the fiber of a hereditary order.
  Namely, the presence of loops in the quiver of $\order{A}(p)$
  implies that any simple $\order{A}(p)$-module has non-trivial first self-extension group,
  which is not the case for a hereditary order.
  The description of the fiber of a hereditary order,
  showing it has no loops,
  can be obtained using the same methods as for the proof of \cref{lemma:quiver-of-terminal-order},
  see also \cite{categorical-absorption},
  or it can be deduced from the description in \cite{MR1351364}.
\end{proof}

To calculate the fiber of a terminal order at a singular point of the ramification,
we work \'etale locally,
and we will recall the standard form from \cite[Definition 2.6]{MR2180454}.
Let $R = \field\{u,v\}$ and $S = R\langle x,y\rangle/(x^{e}-u,y^{e}-v,xy-\zeta yx)$,
where $e\in \mathbb{N}$ and $\zeta$ is an $e$th root of unity.
It follows from the discussion in \cite[\S 2.3]{MR2180454} that \'etale-locally
a terminal order $\order{A}$ is isomorphic to
\begin{equation}
  \Lambda(n,\zeta) =
  \begin{pmatrix}
    S      & S  & \ldots & S  & S \\
    xS     & S  &        & S  & S \\
    \vdots &    & \ddots &    & \vdots \\
    xS     & xS & \ldots & S  & S \\
    xS     & xS & \ldots & xS & S \\
  \end{pmatrix} \quad \subseteq \Mat_n(S)
\end{equation}
for some~$n\in\mathbb{N}$.
This allows us to prove the following.
\begin{lemma}
  \label{lemma:finite-dimensional-algebra-from-terminal-order}
  The $\field$-algebra $\overline{\Lambda} = \Lambda(n,\zeta)\otimes_{R} R/(u,v)$
  is a $n^2 e^2$-dimensional $\field$-algebra,
  with basis spanned by $\{e_{ij}^{(\alpha,\beta)}\mid 1\le i,j\le n, 0\le \alpha,\beta\le e-1\}$
  subject to the multiplication rule
  \begin{equation}
    e_{ij}^{(\alpha,\beta)}\cdot e_{jk}^{(\gamma,\delta)}
    =
    \begin{cases}
      e_{ik}^{(\alpha+\gamma,\beta+\delta)}& \text{if }i\le j\le k, \; j\le k<i,\; k<i\le j,\\
      e_{ik}^{(\alpha+\gamma+1,\beta+\delta)} & \text{otherwise},
    \end{cases}
  \end{equation}
  where we set $e_{ik}^{(\alpha,\beta)} = 0$ if $\alpha\ge e$ or $\beta \ge e$.
  Moreover, we have $e_{ij}^{(\alpha,\beta)}\cdot e_{j^\prime k}^{(\gamma,\delta)} = 0$ if $j\neq j^\prime$.
\end{lemma}

\begin{proof}
  The algebra $\Lambda(n,\zeta)$ is freely generated by the elementary matrices
  multiplied by the element $x^{\alpha}y^{\beta}$ for $0\le \alpha,\beta \le e-1$.
  These form a $\field$-basis of $\overline{\Lambda}$,
  where we denote by $e_{ij}^{(\alpha,\beta)}$ the image of the elementary matrix
  with $x^{\alpha}y^{\beta}$ in the $(i,j)$th entry.
  The multiplication rules follow from the algebra structure of $\Lambda(n,\zeta)$.
\end{proof}

This allows us to obtain the following.

\begin{lemma}
  \label{lemma:quiver-of-terminal-order}
  With the assumptions from \cref{lemma:finite-dimensional-algebra-from-terminal-order},
  the algebra $\overline{\Lambda}$ is isomorphic to the quotient of the path algebra $\field Q/I$,
  where $Q$ is the quiver
  
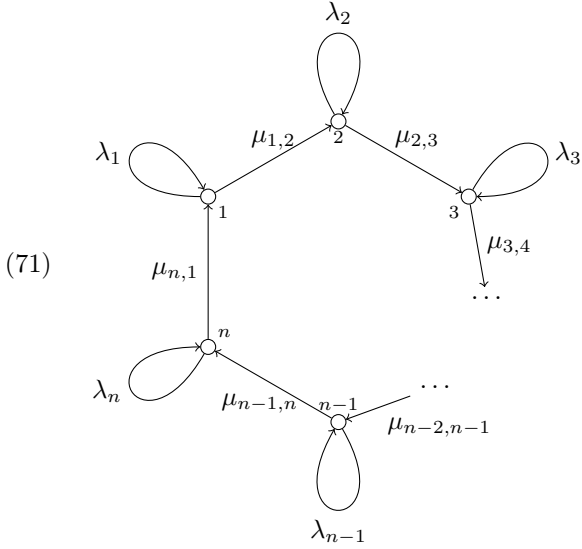
\begin{equation}
    \begin{tikzpicture}[baseline,vertex/.style={draw, circle, inner sep=0pt, text width=2mm}, scale=2.5]
      \node[vertex] (a) at (150:.8cm) {};
      \node[vertex] (b) at (90:.8cm)  {};
      \node[vertex] (c) at (30:.8cm)  {};
      \node         (d) at (350:.8cm) {$\ldots$};
      \node         (e) at (310:.8cm) {$\ldots$};
      \node[vertex] (f) at (270:.8cm) {};
      \node[vertex] (g) at (210:.8cm) {};
      \node[below right] at (a) {$\scriptstyle 1$};
      \node[below]       at (b) {$\scriptstyle 2$};
      \node[below left]  at (c) {$\scriptstyle 3$};
      \node[above]       at (f) {$\scriptstyle n-1$};
      \node[above right] at (g) {$\scriptstyle n$};
      \draw[-{Classical TikZ Rightarrow[]}] (a) -- node [above] {$\mu_{1,2}$\quad} (b);
      \draw[-{Classical TikZ Rightarrow[]}] (b) -- node [above] {\quad$\mu_{2,3}$} (c);
      \draw[-{Classical TikZ Rightarrow[]}] (c) -- node [right] {$\mu_{3,4}$} (d);
      \draw[-{Classical TikZ Rightarrow[]}] (e) -- node [below right] {$\mu_{n-2,n-1}$} (f);
      \draw[-{Classical TikZ Rightarrow[]}] (f) -- node [below] {$\mu_{n-1,n}\quad$} (g);
      \draw[-{Classical TikZ Rightarrow[]}] (g) -- node [left] {$\mu_{n,1}$} (a);
      \draw[-{Classical TikZ Rightarrow[]}] (a) to [out=180,in=120,looseness=40] node [left]  {$\lambda_1$} (a);
      \draw[-{Classical TikZ Rightarrow[]}] (b) to [out=120,in=60,looseness=40]  node [above] {$\lambda_2$} (b);
      \draw[-{Classical TikZ Rightarrow[]}] (c) to [out=60,in=0,looseness=40]    node [right] {$\lambda_3$} (c);
      \draw[-{Classical TikZ Rightarrow[]}] (f) to [out=300,in=240,looseness=40] node [below] {$\lambda_{n-1}$} (f);
      \draw[-{Classical TikZ Rightarrow[]}] (g) to [out=240,in=180,looseness=40] node [left] {$\lambda_n$} (g);
    \end{tikzpicture}
  \end{equation}
  and $I\unlhd \field Q$ is the ideal generated by the relations
  \begin{equation}
    \begin{aligned}
      (\mu_{i,i+1}\cdots \mu_{i-1,i})^{e} \quad &\text{for all }i = 1,\ldots, n,\\
      \lambda_i\mu_{i,i+1}-\mu_{i,i+1}\lambda_{i+1}\quad &\text{for all }i = 1,\ldots, n-1,\\
      \lambda_n\mu_{n,1}-\zeta \mu_{n,1}\lambda_1.\quad &\
    \end{aligned}
  \end{equation}
\end{lemma}

\begin{proof}
  The $\field$-algebra $\overline{\Lambda}$ is finite-dimensional, basic and connected:
  a full set of primitive orthogonal idempotents is formed by $e_{11}^{(0,0)},\ldots, e_{nn}^{(0,0)}$.
  It is a straightforward calculation that $\rad(\Lambda)/\rad^2(\Lambda)$
  is generated by $e_{ii}^{(0,1)}$, $e_{i,i+1}^{(0,0)}$ and $e_{n,1}^{(0,0)}$.
  This explains the shape of the quiver.

  The relations follow from the multiplication rules in \cref{lemma:finite-dimensional-algebra-from-terminal-order}.
\end{proof}

It would be interesting to upgrade these results.
First of all, \cref{proposition:restriction-terminal-order}
should hold for tame orders,
but our proof depends on the local structure of terminal orders for the time being.

Secondly,
our results suggest to study what happens in the singular case.
There exist other order-theoretic properties one can study,
such as being Bass, Gorenstein, tiled, or nodal.
These are defined in \cref{section:noncommutative-conics}.
We will study these notions in explicit examples,
where we have access to additional tools,
These examples indicate what the full picture should look like.

\section{Noncommutative plane curves}
\label{section:noncommutative-plane-curves}
In what follows we will discuss and extend results from the literature
which were obtained through very different methods,
but for which the perspective of central curves gives a unified and improved approach.
These applications are concerned with quotients of three-dimensional Artin--Schelter regular algebras,
thus giving rise to some notion of ``noncommutative plane curve''.

As these Artin--Schelter regular algebras from the literature
are (almost) all finite over the center,
we will in \cref{subsection:central-proj} recall some properties of the central Proj construction,
which turns them into maximal orders on~$\mathbb{P}^2$.
This makes them amenable to the setup of central curves introduced in \cref{section:restriction}.
As explained in the introduction,
we will mostly focus on the case of noncommutative plane conics
and noncommutative plane cubics,
in \cref{section:noncommutative-conics} resp.~\cref{section:noncommutative-cubics}.

%
%

\subsection{Reminder on central Proj construction}
\label{subsection:central-proj}
To a graded algebra which is finite over its center
one can associate an order on a projective variety via the central Proj construction,
which we will recall from \cite[Section 2]{MR1356364}.

Let~$A$ be a connected graded algebra generated in degree $1$,
which is moreover finite as a module over a central graded integral domain~$R\subset A$.
Then there is a projective variety~$Y = \Proj R$
and~$A$ induces a coherent sheaf of algebras~$\order{A} \in \coh Y$,
defined on each standard open~$\distinguished(g)\subset Y$ by
\begin{equation}
  \label{eq:defofalgebraonaffineopen}
  \Gamma(\distinguished(g),\order{A}) \colonequals A[g^{-1}]_0,
\end{equation}
with the~$R[g^{-1}]_0$-linear multiplication induced by the one on~$A$.
In general~$\order{A}$ is not an~$\mathcal{O}_Y$-order, even for the choice $R=\ZZ(A)$,
and its center~$\mathcal{Z} = \ZZ(\order{A})$ is in general bigger than~$\mathcal{O}_Y \subset \mathcal{Z}$.
However, one can consider the affine morphism
\begin{equation}
  \label{equation:central-proj-affine}
  s\colon \relSpec_Y \mathcal{Z} \longrightarrow Y.
\end{equation}
and there is a sheaf of algebras~$\order{\underline A} \in \coh \relSpec_Y \mathcal{Z}$
that is uniquely defined by the property~$s_*\order{\underline A} = \order{A}$.
The pair~$(\relSpec_Y \mathcal{Z},\order{\underline A})$ is called the \emph{central Proj} of~$A$.
In \cite{MR1356364} the sheaf of algebras~$\order{\underline A}$ was somewhat erroneously denoted~$s^*\order{A}$.

It was shown in~\cite{MR1356364} that $\order{\underline A}$ is a well-behaved order
under some regularity conditions on $A$.
In particular we have the following result \cite[Proposition~1]{MR1356364},
which will apply in the setup of \cref{section:dictionaries,section:restriction}.

\begin{proposition}[Le Bruyn]
  \label{proposition:le-bruyn-maximal order}
  If~$\gldim A < \infty$
  then~$\order{\underline A}$ is a maximal order over~$\relSpec_Y \mathcal{Z}$.
  Moreover,~$\relSpec_Y \mathcal{Z}$ is normal and Cohen--Macaulay,
  and~$\order{\underline A}$ is a sheaf of Cohen--Macaulay modules.
\end{proposition}

The noncommutative geometry of~$A$ is captured by both~$(Y,\order{A})$
and the central Proj~$(\relSpec_Y \mathcal{Z},\order{\underline A})$
since there are equivalences
\begin{equation}
  \label{equation:central-proj-equivalences}
  \qgr A \simeq \coh(Y,\order{A}) \simeq \coh(\relSpec_Y\mathcal{Z},\order{\underline A}).
\end{equation}

There are various algebras~$A$ for which the central Proj is simply isomorphic to $Y = \Proj\ZZ(A)$:
it was shown by Smith--Tate \cite[Lemma 5.1]{MR1273835} that this is true if the center $\ZZ(A)$
is generated by non-zerodivisors of coprime degrees.
Moreover, for an algebra~$A$ with generators~$z_1,\ldots, z_n$ in degree~$d_i$
such that~$d = \gcd(d_1,\ldots, d_n)$,
one can employ the \emph{Veronese construction}
\begin{equation}
  A^{(d)} = \bigoplus_{n\in\mathbb{Z}} A^{(d)}_n = \bigoplus_{n\in\mathbb{Z}} A_{nd}.
\end{equation}
As before we define the space $Y^{(d)} = \Proj \ZZ(A^{(d)})$
and a sheaf of algebras $\order{A}^{(d)}$ on $Y^{(d)}$ associated to $A^{(d)}$.
Keeping track of the noncommutative sheaf of algebras, we get the following.

\begin{lemma}[{After \cite[Lemma 5.1]{MR1273835}}]
  \label{lemma:veronese-construction}
  Let $A$ be a graded algebra over $\field$ which is finite over its center, and suppose the center is generated by homogeneous elements $z_1,\ldots,z_n$ of degrees $d_1,\ldots,d_n$ which are non-zerodivisors.
  Writing $d = \gcd(d_1,\ldots,d_n)$ there is an isomorphism of ringed spaces
  \begin{equation}
    (\relSpec_Y\mathcal{Z},\order{\underline A}) \cong (Y^{(d)}, \order{A}^{(d)}).
  \end{equation}
\end{lemma}

\begin{proof}
  In \cite[Lemma 5.1]{MR1273835} Smith and Tate show that~$\ZZ(A^{(d)}[z_i^{-1}]_0) = \ZZ(A[z_i^{-1}]_0)$
  under the given assumptions, thereby proving~$Y^{(d)} \cong \relSpec_Y\mathcal{Z}$.
  It remains to show that~$\order{\underline A}$ corresponds to~$\order{A}^{(d)}$, but this is simply the identification
  \begin{equation}
    \Gamma(\distinguished(z_i),\order{\underline A}) = A[z_i^{-1}]_0 = A^{(d)}[z_i^{-1}]_0 = \Gamma(\distinguished(z_i),\order{A}^{(d)}),
  \end{equation}
  which is also given in the proof of \cite[Lemma 5.1]{MR1273835}.
\end{proof}

Hence we obtained a chain of equivalences
\begin{equation}
  \qgr A \simeq \coh(Y,\order{A}) \simeq \coh(\relSpec_Y\mathcal{Z},\order{\underline A}) \simeq \coh(Y^{(d)},\order{A}^{(d)}) \simeq \qgr A^{(d)}.
\end{equation}
From this point on
we will no longer distinguish notation for~$\order{\underline A}$ and~$\order{A}$,
rather we will just write~$\order{A}$
for the maximal order on~$\relSpec_Y\mathcal{Z}$,
as~$Y$ will no longer play a role.

\paragraph{The central Proj of noncommutative projective planes}
In the remainder of this paper we will consider the special case of
quadratic 3-dimensional AS-regular algebras,
which are coordinate rings of noncommutative projective planes.
In this case the central Proj construction gives a particularly pleasant result,
which was obtained in full generality by Van Gastel~\cite[Proposition~4.2]{MR1880659},
with earlier work by Artin for the Sklyanin case~\cite[Theorem~5.2]{MR1144023},
and Mori for the skew case~\cite[Theorem~4.7]{MR1624000}.
\begin{theorem}
  \label{theorem:van-gastel}
  Let~$A$ be a quadratic 3-dimensional Artin--Schelter regular algebra
  which is finite over its center.
  Then the central Proj of~$A$ is a pair~$(\mathbb{P}^2,\order{A})$
  with~$\order{A}$ a maximal order ramified along a cubic curve which is of one of the following types:
  \begin{enumerate}
    \item[(A)] elliptic curve
    \item[(B)] nodal cubic
    \item[(D)] union of a conic and a line in general position
    \item[(E)] three lines in general position
  \end{enumerate}
\end{theorem}
The choice of letters will be explained in \cref{subsection:graded-clifford-algebras}.

\subsection{Noncommutative plane curves using central curves}
\label{subsection:noncommutative-via-central-curves}
If~$A$ is any graded algebra which is finite over its center
as in \cref{subsection:central-proj}
and~$f\in A$ is a homogeneous central element,
then the quotient algebra
\begin{equation}
  B\colonequals A/(f).
\end{equation}
defines a noncommutative projective variety~$\qgr B$ that can be viewed as a hypersurface in~$\qgr A$.
We will consider the special case where~$\qgr A$ is a noncommutative plane.

\begin{definition}
  \label{definition:noncommutative-plane-curve}
  Let~$A$ be a quadratic 3-dimensional Artin--Schelter regular algebra
  which is finite over its center.
  Let~$f\in\ZZ(A)_d$ be a central homogeneous element of degree~$d$.
  The noncommutative projective variety~$\qgr A/(f)$ is a \emph{noncommutative plane curve}.
\end{definition}
As explained above,
the literature contains results for noncommutative plane curves when~$d=2,3$,
which we will revisit using the machinery of central curves.

In the rest of this section we fix~$A$ to be a quadratic 3-dimensional Artin--Schelter regular algebra
which is finite over its center~$R\colonequals\ZZ(A)$.
By \cref{lemma:veronese-construction} and \cref{theorem:van-gastel}
the central Proj is of the form
\begin{equation}
  (\mathbb{P}^2,\order{A})
  \cong
  (\Proj\ZZ(A^{(e)}),\order{A}^{(e)}),
\end{equation}
where~$e$ denotes the~$\gcd$ of the degrees of the homogeneous generators of~$R$.
We moreover fix a homogeneous central element~$f\in R_d$
and consider the noncommutative plane curve~$\qgr B$ defined by
\begin{equation}
  B \colonequals A/(f).
\end{equation}
The following reduces the study of the noncommutative plane curve
to that of a sheaf of algebras on a curve,
not via another application of the central Proj construction,
but by restricting the output of the central Proj construction for~$A$.
\begin{proposition}
  \label{proposition:restriction-central-proj}
  Let~$C=\VV(f)\subset\mathbb{P}^2$ be the curve defined by~$f\in R_d \subset \ZZ(A^{(e)})$,
  and define~$\order{B}\colonequals\order{A}|_C$.
  Then there exists an equivalence of categories
  \begin{equation}
    \qgr B
    \simeq
    \coh(C,\order{B}).
  \end{equation}
\end{proposition}

Before we prove this,
we first discuss 2 preliminary (and standard) lemmas.
For the first we include a proof for completeness' sake,
for the second we will prove a stacky generalization in \cref{appendix:lemmas}.
\begin{lemma}
  With~$A,B,f$ as in \cref{proposition:restriction-central-proj},
  the category~$\qgr B$ embeds into~$\qgr A$ as the subcategory
  \begin{equation}
    \{ M \in \qgr A \mid M \text{ is~$f$-torsion} \} \subset \qgr A.
  \end{equation}
  Likewise,
  the category~$\qgr A^{(e)}/fA^{(e)}$
  embeds into~$\qgr A^{(e)}$
  as the subcategory of~$f$-torsion modules.
\end{lemma}

\begin{proof}
  Observe that the category~$\gr B$ embeds into~$\gr A$ as the full subcategory of~$f$-torsion modules.
  Taking the quotient by the Serre subcategory~$\fd A \subset \gr A$, the composed functor
  \begin{equation}
    \gr B \to \gr A/\fd A = \qgr A,
  \end{equation}
  has kernel~$\fd B$, and therefore yields an embedding of~$\qgr B = \gr B/\fd B$ into~$\qgr A$
  via the universal property of the Serre quotient.

  The proof for the Veronese subalgebra is similar.
\end{proof}

\begin{lemma}
  \label{lemma:restricted-as-torsion}
  The category~$\coh(C,\order{B})$ embeds into~$\coh(S,\order{A})$ as the subcategory
  \begin{equation}
    \{ \mathcal{F} \in \coh(S,\order{A}) \mid \mathcal{I}\mathcal{F} = 0 \} \subset \coh(S,\order{A}),
  \end{equation}
  where~$\mathcal{I}\subset \mathcal{O}_S$ denotes the sheaf of ideals generated by~$f\in\ZZ(A^{(e)})$.
\end{lemma}

\begin{proof}
  This is a standard result,
  alternatively one can consider it a special case of \cref{lemma:embedding-as-torsion}.
\end{proof}

\begin{proof}[Proof of \cref{proposition:restriction-central-proj}]
  The Veronese equivalence~$\qgr A \simeq \qgr A^{(e)}$,
  which exists because~$A$ is generated in degree~1,
  preserves the~$f$-torsion objects,
  hence it suffices to show that the equivalence~$\qgr A^{(e)} \simeq \coh(S,\order{A})$
  identifies~$f$-torsion objects with~$\mathcal{I}$-torsion objects,
  so that we can conclude by \cref{lemma:restricted-as-torsion}.

  Suppose~$M \in \qgr A^{(e)}$ is~$f$-torsion, and let~$\mathcal{M} \in \coh(X,\order{B})$ be the corresponding sheaf.
  Then for any standard open~$D(g) \subset X$,
  and any local section~$\frac{fa}{g^n} \in \mathcal{I}(\mathrm{D}(g)) = (f\cdot\ZZ(A^{(e)})[g^{-1}])_0$ of~$\mathcal{I}$,
  we have
  \begin{equation}
    \frac{fa}{g^n} \cdot \mathcal{M}(D(g))
    =\frac{fa}{g^n}\cdot (M[g^{-1}])_0
    =\left(\frac{fa}{g^n}\cdot M[g^{-1}]\right)_0
    =0.
  \end{equation}
  Hence~$\mathcal{M}$ is a~$\mathcal{I}$-torsion sheaf.
  Conversely, let~$\mathcal{M} \in \coh(X,\order{B})$ be a~$\mathcal{I}$-torsion sheaf.
  Then the corresponding object of~$\qgr A^{(e)}$ is isomorphic to
  \begin{equation}
    \Gamma_*(\mathcal{M}) \in \qgr \Gamma_*(\order{B}).
  \end{equation}
  Since~$\mathcal{M}$ is~$\mathcal{I}$-torsion,
  it follows that~$\Gamma_*(\mathcal{M})$ is torsion for the ideal~$\Gamma_*(\mathcal{I}) \subset \Gamma_*(\order{B})$.
  Viewing~$\Gamma_*(\mathcal{M})$ now as an~$A^{(e)}$-module
  via the canonical map~$A^{(e)} \to \Gamma_*(\order{B})$,
  we find that~$f$ acts via the global section~$f \in \Gamma(\mathcal{I}(n)) \subset \Gamma_*(\mathcal{I})$
  with~$n = \deg f$.
  Hence~$\Gamma_*(\mathcal{M})$ is an~$f$-torsion module, which shows the equivalence.
\end{proof}

\section{Noncommutative conics}
\label{section:noncommutative-conics}
The first interesting example of noncommutative plane curves
is given by noncommutative conics,
i.e.,
we consider a~3-dimensional Artin--Schelter regular algebra~$A$
and a central element~$f\in\ZZ(A)_2$ .
These noncommutative conics have been studied and classified in \cite{MR4531545,MR4794245}
(with a special case already appearing in \cite{MR4575408}).

\begin{itemize}
  \item In \cite{MR4531545} the starting point is that of 3-dimensional Calabi--Yau Artin--Schelter regular algebras
    admitting at least one central element of degree~2,
    which are subsequently classified in Section~3 of op.~cit.
    All but one of these is finite-over-the-center.

  \item In \cite{MR4794245} on the other hand,
    the starting point is that of graded Clifford algebras.
    These naturally admit a net of central elements of degree~2,
    and are finite-over-their-center,
    making them suitable for the methods introduced in \cref{section:noncommutative-plane-curves}.
    In \cref{remark:cy-vs-clifford} we explain how these two choices of
    ambient Artin--Schelter regular algebras compare.
\end{itemize}

The methods introduced in \cref{subsection:noncommutative-via-central-curves}
give a new perspective on noncommutative conics.
In \cref{subsection:clifford-conics-orders} we will describe them
using sheaves of algebras,
and in \cref{subsection:clifford-conics-stacks} we will describe them
using the geometry of stacks.
In \cref{subsection:clifford-conics-comments}
we will (re)classify noncommutative conics
from our perspective,
compare our description to the existing descriptions in the literature,
and obtain some corollaries.

\subsection{Graded Clifford algebras and Clifford conics}
\label{subsection:graded-clifford-algebras}
A natural source for maximal orders on~$\mathbb{P}^n$
is given by the central Proj construction applied to graded Clifford algebras,
as studied in \cite[Section 4]{MR1356364}.
We will soon specialise to~$n=2$,
but for now we will work in the general case.

We start by recalling the definition of a graded Clifford algebra.
Let~$R = \field[z_0,\ldots,z_n]$ be the polynomial ring in~$n+1$ variables with~$\deg z_i = 2$.
Consider~$M \in \Mat_{n+1}(R_2)$ a symmetric matrix with linear entries in the~$z_i$.
It is useful to write
\begin{equation}
  M = M_0 z_0 + \ldots + M_{n} z_n,
\end{equation}
where~$M_i \in \Mat_{n+1}(\field)$ are symmetric matrices.

The \emph{graded Clifford algebra} associated to~$M$ is defined as~$\Cl(M) = R\langle x_0, \ldots, x_n\rangle/I$
where~$I$ is the two-sided ideal generated by
\begin{equation}
  x_i x_j + x_j x_i = M_{i,j} = \sum_{k= 0}^n (M_k)_{i,j} z_k,\quad 0\le i,j \le n.
\end{equation}
By setting~$\deg x_i = 1$, this becomes a graded algebra generated in degree 1 and 2
which is finitely generated as a module over the central subring~$R$.
To study the algebraic properties of~$\Cl(M)$ it is useful to consider its associated geometry.

\paragraph{A linear system of quadrics associated to a graded Clifford algebra.}
Each of the matrices~$M_i$ defines via its associated quadratic form
a quadric hypersurface~$Q_i\subseteq \mathbb{P}^n$.
Therefore we can view the data~$M_i$ as an~$n$-dimensional linear system
\begin{equation}
  \mathcal{Q}_M = \field Q_0 + \ldots + \field Q_n \subset \HH^0(\mathbb{P}^n,\mathcal{O}_{\mathbb{P}^n}(2)).
\end{equation}
The geometry of this linear system tells us
when the graded Clifford algebra can be considered
as a noncommutative projective space.
\begin{proposition}[{\cite[Proposition 7]{MR1356364}}]
  The graded Clifford algebra~$\Cl(M)$ is an Artin--Schelter regular~$\field$-algebra
  if and only if~$\mathcal{Q}_M$ is basepoint-free.
\end{proposition}

Since~$R\subseteq \Cl(M)$ is central,
the first step of the central Proj construction
from \cref{subsection:central-proj}
can be carried out over~$\mathbb{P}^n = \Proj(R)$.
Denote by~$(\mathbb{P}^n,\order{A})$ the corresponding noncommutative space
and by~$s\colon \relSpec_{\mathbb{P}^n}\mathcal{Z}\to \mathbb{P}^n$
the central Proj as defined in \eqref{equation:central-proj-affine}.

If the associated linear system of quadrics~$\mathcal{Q}_M$ is basepoint-free,
the~$\mathcal{O}_{\relSpec_{\mathbb{P}^n}\mathcal{Z}}$-order~$\order{A}$
is maximal by \cite[Proposition 1]{MR1356364}.

Recall that the \emph{discriminant} of the linear system~$\mathcal{Q}_M$
is the closed subscheme~$\Delta(\mathcal{Q}_M)\subset \mathbb{P}^n$,
where the quadrics are singular.
The following lemma explains what the central Proj looks like for~$n$ even.

\begin{lemma}
  \label{lemma:ramification}
  Assume that~$n$ is even
  and~$\mathcal{Q}_M$ is basepoint-free.
  Then $\relSpec_{\mathbb{P}^n}\mathcal{Z}\cong \mathbb{P}^n$
  and the ramification divisor of the~$\mathcal{O}_{\mathbb{P}^n}$-order $\order{A}$
  is given by
  \begin{equation}
    S_1 \colonequals \{p \in \mathbb{P}^n\mid \rk M(p) \le n\}.
  \end{equation}
\end{lemma}

\begin{proof}
  Since~$n$ is even, the center~$\ZZ(\Cl(M))$ is~$R \oplus R \delta = \field[z_0,\ldots, z_n, \delta]$,
  where $\delta \in \Cl(M)$ is an element of degree $n+1$ such that $\delta^2 = \det M$.
  A proof for this result can be found in \cite{MR741798},
  or more recently in \cite[Theorem 3.18]{MR4794245}.

  From \cref{lemma:veronese-construction} it follows that
  the underlying variety of the central Proj construction for
  the graded Clifford algebra
  is given by~$\Proj(\ZZ(\Cl(M))^{(2)}) = \Proj(\field[z_0,\ldots, z_n])\cong \mathbb{P}^{n}$.
  By \cite[Proposition 9]{MR1356364} the ramification divisor of the associated~$\mathcal{O}_{\mathbb{P}^n}$-order~$\order{A}$
  has ramification~$\Delta=S_1$.
\end{proof}

\paragraph{Clifford conics}
In the case~$n=2$, we obtain a net of conics as our linear system in the construction of a graded Clifford algebra.
By \cite[Table 2]{MR0432666}, for a basepoint-free net of conics~$\mathcal{Q}_M$,
the discriminant must be one of the following (keeping the notation from op.~cit.):
\begin{enumerate}
  \item[(A)] an elliptic curve,
  \item[(B)] a nodal cubic,
  \item[(D)] a conic meeting a line in general position, or
  \item[(E)] three lines in general position.
\end{enumerate}

The classification of nets of conics is sufficient for the description of
the ramification divisor of~$\mathcal{O}_{\mathbb{P}^2}$-orders obtained from
the central Proj construction applied to graded Clifford algebras.
More precisely, we have the following,
where an order on a smooth surface is said to be \emph{terminal} as in \cite[Definition~2.5]{MR2180454}
if its discriminant~$\Delta$ has normal crossings,
with the cyclic covers of the components ramified at the nodes,
and at each node one cyclic cover is totally ramified of degree~$e$
with the other being ramified with index~$e$ and degree~$ne$ for some~$n\geq 2$.
This will allow the use of \'etale local models for terminal orders \cite[Section~2.3]{MR2180454}.

\begin{lemma}
  \label{lemma:terminal}
  The~$\mathcal{O}_{\mathbb{P}^2}$-order~$\order{A}$
  constructed in \cref{lemma:ramification}
  is terminal
  and its ramification divisor is $\Delta = \Delta(\mathcal{Q}_M)$.
\end{lemma}

This explains the labeling of the discriminants in \cref{theorem:van-gastel}.

\begin{proof}
  The ramification of the central Proj of a graded Clifford algebra was determined in \cref{lemma:ramification}.
  It remains to show that the~$\mathcal{O}_{\mathbb{P}^2}$-order is terminal.
  Since $\mathcal{Q}_M$ is basepoint-free by assumption,
  the ramification divisor is a normal crossing divisor.

  Given a point~$p\in \mathbb{P}^{2}$ of codimension 1,
  the~$\field(p)$-algebra $\order{A}_p/\rad\order{A}_p$ is simple artinian
  and~$\rad\order{A}_p$ is principal. 
  Therefore, the quotient~$\order{A}_p/\rad\order{A}_p$ is
  either isomorphic to $\Mat_2(\field(p))$,
  since~$\field$ is algebraically closed and thus Tsen's theorem applies to~$\field(p)$,
  or to a field extension $L/\field(p)$ of degree 2.
  To see this, one can use the classification of quaternion~$\field$-algebras of \cite[Theorem 4.5]{MR2905553}.

  In the first case the~$\mathcal{O}_{\mathbb{P}^2,p}$-order is Azumaya and hence unramified.
  In the second case, the ramification index is~$e_p = 2$
  and we obtain a two-fold covering over each irreducible component of the ramification divisor
  which is ramified over the points of intersection of the irreducible components.
  Because we only have double covers,
  the~$\mathcal{O}_{\mathbb{P}^2}$-order is terminal.
\end{proof}

For more on the geometry of graded Clifford algebras from the perspective of sheaves of orders,
one is referred to \cite{1811.08810}.

Because 3-dimensional graded Clifford algebras have a net of central elements of degree~2,
we can define noncommutative conics.
We introduce the following terminology.

\begin{definition}
  \label{definition:clifford-conic}
  Let~$A$ be the 3-dimensional graded Clifford algebra
  associated to the basepoint-free net of conics~$\mathcal{Q}_M$.
  Let~$f\in\ZZ(A)_2=\field z_0+\field z_1+\field z_2$ be a central element of degree~2,
  and set~$B\colonequals A/(f)$.
  We call the abelian category~$\qgr B$ a \emph{Clifford conic}.
\end{definition}

This choice of terminology sets them apart from other sources of noncommutative conics
studied in the literature.
The first class of noncommutative conics in the literature are ``skew conics''.
These are defined with respect to
a skew polynomial algebra~$\field\langle x_0,x_1,x_2\rangle/(x_ix_j-q_{i,j}x_jx_i)$
where~$q_{i,j}=\pm1$ \cite{MR4575408}.
The second source are noncommutative conics inside Calabi--Yau quantum projective planes \cite{MR4531545}.
We will compare these classes of noncommutative conics
to Clifford conics in \cref{subsection:clifford-conics-comments}.

\begin{remark}
  \label{remark:cy-vs-clifford}
  By \cite[Theorem~1.2]{MR4794245} a 3-dimensional graded Clifford algebra is Calabi--Yau.
  Conversely, the 4~relevant isomorphism types of graded Calabi--Yau algebras with a central net of conics
  from \cite[Theorem~3.6]{MR4531545}
  are graded Clifford algebras, by using the symmetric matrices
  \begin{equation}
    \begin{pmatrix}
      2z_0 & 0 & 0 \\
      0 & 2z_1 & 0 \\
      0 & 0 & 2z_2
    \end{pmatrix},
    \begin{pmatrix}
      2z_0 & 0 & 0 \\
      0 & 2z_1 & -z_0 \\
      0 & -z_0 & 2z_2
    \end{pmatrix},
    \begin{pmatrix}
      2z_0 & 0 & -z_1 \\
      0 & 2z_1 & -z_0 \\
      -z_1 & -z_0 & 2z_2
    \end{pmatrix},
    \text{resp.~}
    \begin{pmatrix}
      2z_0 & -\lambda z_2 & -\lambda z_1 \\
      -\lambda z_2 & 2z_1 & -\lambda z_0 \\
      -\lambda z_1 & -\lambda z_0 & 2z_2
    \end{pmatrix},
  \end{equation}
  in the order of \cite{MR4531545}.
\end{remark}

\subsection{Properties of orders on curves}
\label{subsection:properties-orders}
In order to understand non-smooth noncommutative curves,
we recall the notion of Gorenstein and Bass orders.
Denote by
\begin{equation}
  \omega_{\order{A}} = \intHom_{\mathcal{O}_Y}(\order{A},\mathcal{O}_Y)
\end{equation}
the dualizing bimodule of the $\mathcal{O}_Y$-order~$\order{A}$.
Moreover, let~$\order{A}_\eta$ be the localization of~$\order{A}$ at the generic point~$\eta$ of~$Y$. Note that~$\order{A}_\eta$ is always isomorphic to a matrix algebra over the function field of~$Y$ since~$\field$ is algebraically closed.
Following \cite{MR4199211},
we give the following definition of Gorenstein and Bass orders.
\begin{definition}
  Let~$\order{A}$ be an~$\mathcal{O}_Y$-order.
  \begin{enumerate}
    \item We say that~$\order{A}$ is \emph{Gorenstein}
      if the dualizing bimodule~$\omega_{\order{A}}$ is locally free as an~$\order{A}$-bimodule.
    \item If every order~$\order{A}^\prime\supset \order{A}$ in~$\order{A}_\eta$
      containing~$\order{A}$ is Gorenstein,
      we say that~$\order{A}$ is \emph{Bass}.
  \end{enumerate}
\end{definition}

Nodal orders have been studied in a series of papers \cite{burban2024aspectstheorynodalorders, burban2019noncommutativenodalcurvesderived,MR2029026}.
To define a \emph{nodal order} globally we first need to recall the local definition from \cite[Definition 3.1]{burban2024aspectstheorynodalorders}.

Let $R$ be a discrete valuation ring and $\Lambda \subset \text{Mat}_n(\Frac(R))$ be an $R$-order. We say that $\Lambda$ is \emph{nodal} if it is contained in a hereditary $R$-order $\Lambda^\prime$ in $\text{Mat}_n(\Frac(R))$ such that $\rad(\Lambda^\prime) = \rad(\Lambda)$ and for every finitely generated simple left $\Lambda$-module $S$ the length of $\Lambda^\prime\otimes_\Lambda S$ as a $\Lambda$-module is bounded by 2.

With this setup we are able to define nodal orders on a curve, see \cite[Definition 4.1]{burban2024aspectstheorynodalorders}.
\begin{definition}
  An $\mathcal{O}_Y$-order~$\order{A}$ is \emph{nodal} if for every point~$p\in Y$ the $\mathcal{O}_{Y,p}$-order~$\order{A}_p$ is nodal.
\end{definition}

For an order~$\order{A}$ on a curve~$Y$ we thus have the following implications
\begin{equation}
  \order{A}\text{ maximal}
  \quad\Rightarrow\quad\order{A}\text{ hereditary}
  \quad\Rightarrow\quad\order{A}\text{ nodal}
  \quad\Rightarrow\quad\order{A}\text{ Bass}
  \quad\Rightarrow\quad\order{A}\text{ Gorenstein.}
\end{equation}

We note that these four properties of orders can be checked locally.
\begin{lemma}
  \label{lemma:order-properties-local}
  Let~$\order{A}$ be an order on a smooth curve~$Y$.
  The following are equivalent:
  \begin{enumerate}
    \item The~$\mathcal{O}_Y$-order~$\order{A}$ is maximal, resp.~hereditary, Bass, or~Gorenstein.
    \item For each~$p\in Y$ the~$\mathcal{O}_{Y,p}$-order~$\order{A}_p$ is maximal, resp.~hereditary, Bass, or~Gorenstein.
    \item For each~$p\in Y$ the~$\widehat{\mathcal{O}}_{Y,p}$-order~$\widehat{\order{A}}_p$ is maximal, resp.~hereditary, Bass, or~Gorenstein.
  \end{enumerate}
\end{lemma}

\begin{proof}
  The equivalence (1) $\Leftrightarrow$ (2) can be found in the literature, see, e.g.~\cite[Lemma 10.4.3 and \S21.4.4, 24.2.2, and~24.5.2]{MR4279905}.
  The equivalence (1) $\Leftrightarrow$ (3) is given in \cite[Corollary (11.6)]{MR0393100} for maximal orders and in \cite[Theorem (40.5)]{MR0393100} for hereditary orders.

  The equivalence (2) $\Leftrightarrow$ (3) for Gorenstein and Bass orders follows from the inclusion-preserving one-to-one correspondence \cite[Theorem (5.2)]{MR0393100} between lattices over a discrete valuation ring and its completion.
\end{proof}

\paragraph{Properties of quaternion orders}
In the following we analyse the complete-local structure of the orders
at points~$p\in \Delta_L$ in the discriminant
arising in \cref{subsection:clifford-conics-orders}.
Fixing~$p \in \Delta_L$ and taking~$R = \field[[t]] = \widehat{\mathcal{O}}_{L,p}$,
we will repeatly write
\begin{equation}
  (t^{e_1},t^{e_2})_2^{R} \colonequals R\langle i,j \rangle/(i^2-t^{e_1}, j^2-t^{e_2}, ij+ji)
\end{equation}
for the quaternion~$R$-order
in $\Mat_2(\field((t)))$ for $e_1,e_2\in \mathbb{Z}_{\ge 0}$,
where $e_1\le e_2$.
Note that
\begin{equation}
  (t^{e_1},t^{e_2})_2^{R}\cong(t^{f_1},t^{f_2})_2^{R}
\end{equation}
if and only if~$(e_1,e_2) = (f_1,f_2)$.
The following describes the local properties of these orders.

An interesting (complete) local property of an order in a matrix algebra is whether it is tiled.
Recall from \cite[\S 2]{MR274527} that an~$R$-order~$\Lambda$ in $\text{Mat}_n(\field((t)))$ is \emph{tiled}
if it has a complete set of $n$ primitive orthogonal idempotents.

\begin{lemma}
  \label{lemma:clifford-conics-local}
  Let $e_{1}, e_2 \ge 0$ and denote by $\Lambda\colonequals(t^{e_1},t^{e_2})_2^{R}$ the cyclic quaternion~$R$-order.
  \begin{enumerate}[label = \roman*)]
    \item The~$R$-order $\Lambda$ is maximal
      if and only if~$e_1 =e_2 = 0$.
    \item The~$R$-order $\Lambda$ is hereditary
      if and only if~$e_1 = 0$ and~$e_2 \le 1$.
    \item The~$R$-order $\Lambda$ is Bass
      if and only if~$e_1 \le 1$.
    \item The~$R$-order $\Lambda$ is Gorenstein
      for all $e_1,e_2 \in \mathbb{Z}_{\ge 0}$.
      The canonical bundle of the order is given by
      \begin{equation}
        \omega_{\Lambda} \cong R \oplus (t^{-e_1})\cdot i \oplus (t^{-e_2})\cdot j \oplus (t^{-e_1-e_2})\cdot ji.
      \end{equation}
      Moreover, the canonical bundle satisfies $\omega_{\Lambda}^2 = (t^{-e_1-e_2})\Lambda$.
    \item The~$R$-order $\Lambda$ is nodal if and only if~$(e_1,e_2)\in \{(0,0), (0,1), (0,2), (1,1)\}$.
    \item The~$R$-order $\Lambda$ is tiled if and only if~$e_1 = 0$.
    \item At the closed point we have
      \begin{equation}
        \Lambda/\rad \Lambda
        \cong
        \begin{cases}
          \Mat_2(\field) & \text{if }e_1 = e_2 = 0,\\
          \field\times \field & \text{if }0= e_1 < e_2,\\
          \field & \text{if }0< e_1 \le e_2.\\
        \end{cases}
      \end{equation}
  \end{enumerate}
\end{lemma}

\begin{remark}
  The $R$-order satisfying $\Lambda/\rad\Lambda \cong \field\times \field$ are also called \emph{residually split}, whereas the $R$-orders satisfying $\Lambda/\rad\Lambda \cong \field$ are called \emph{residually ramified}, see \cite[Definition 24.3.2]{MR4279905}. We list these local properties in \cref{table:local-properties-clifford-conic-orders}.
\end{remark}

\begin{proof}
  The first two statements follow from the classification of
  maximal (resp.~hereditary) orders over complete discrete valuation rings,
  see \cite[Theorem (17.3), Theorem (39.14)]{MR0393100}.

  For the next two statements,
  one can use the functorial one-to-one correspondence \cite[Corollary 3.20]{MR2824785}
  between nondegenerate quadratic forms and ternary quadratic modules.
  The associated bilinear form to $\Lambda$ is given by
  \begin{equation}
    Q\colon R^{\oplus 3}\to R,\quad (\alpha_1,\alpha_2,\alpha_3)\mapsto -t^{e_2}\alpha_1^2 - t^{e_1}\alpha_2^2 +\alpha_3^.
  \end{equation}
  We see that the quadratic form is always primitive,
  hence~$\Lambda$ is Gorenstein, by \cite[Theorem 24.2.10]{MR4279905}. 
Moreover, combining (vii) and \cite[Exercise 24.13(a)]{MR4279905}, statement (iii) follows.

  In (iv), the canonical bundle can be calculated
  by using the~$\Lambda\mhyphen\Lambda$-bimodule isomorphism
  \begin{equation}
    \codiff(\Lambda)\xrightarrow{\sim} \omega_\Lambda,\quad \alpha\mapsto \textsf{trd}(\alpha\cdot -),
  \end{equation}
  where the \emph{codifference} is defined as~$\codiff(\Lambda)\colonequals \{\alpha\in \Mat_2(\field((t)))\mid \trd(\alpha\Lambda)\subset R\}$.

  Statement (v) follows from the classification \cite[Theorem 3.19]{burban2019noncommutativenodalcurvesderived} of nodal orders over $R$ up to conjugation.
  Following \cite[\S 3.2]{burban2019noncommutativenodalcurvesderived} we identify four possible triples $(\Omega,\approx,\sigma)$,
  where $\approx$ is a symmetric binary relation on $\Omega$,
  and $\sigma\in S_{|\Omega|}$ is a cycle of length $|\Omega|$.
  The cycle $\sigma$ has to be of maximal length because the nodal order $\Lambda$ is an order in $\text{Mat}_2(\field((t)))$.
  The cardinality is $|\Omega|\le 2$ by rank reasons.
  Finally, we see that in the case where $|\Omega| =2$ reflexitivity is not allowed, again by rank reasons (on the hereditary overorder).
  Therefore we get the following correspondence:
  \begin{itemize}
    \item $\Omega = \{1\}$ corresponds to $(e_1,e_2) = (0,0)$,
    \item $\Omega = \{1\}$ with $1\approx 1$ corresponds to the $(e_1,e_2) = (0,2)$,
    \item $\Omega = \{1,2\}$ corresponds to $(e_1,e_2) = (0,1)$, and
    \item $\Omega = \{1,2\}$ with $1\approx 2$ corresponds to $(e_1,e_2) = (1,1)$.
  \end{itemize}

  The sixth statement follows from the fact that tiled quaternion~$R$-orders are of the form
  \begin{equation}
    \begin{pmatrix}
      R & R\\
      (a) & R
    \end{pmatrix}
  \end{equation}
  for~$a\in R$ non-zero.

  The last statement follows from a straightforward calculation of the radical.
\end{proof}

\subsection{Clifford conics as orders}
\label{subsection:clifford-conics-orders}
We will now describe Clifford conics as defined in \cref{definition:clifford-conic}
from the perspective of orders
by applying \cref{proposition:restriction-central-proj}.
We thus consider a graded Clifford algebra associated to a basepoint-free net of conics in~$\mathbb{P}^2$,
and we will let~$(\mathbb{P}^2,\order{A})$ denote the central Proj of
the graded Clifford algebra~$\Cl(M)$ over $R = \field[z_0,z_1,z_2]$,
where~$M$ defines the basepoint-free net~$\mathcal{Q}_M$ of conics.
\begin{definition}
  \label{definition:clifford-conic-order}
  Let~$(\mathbb{P}^2,\order{A})$ be the central Proj of a graded Clifford algebra.
  Let $L\cong \mathbb{P}^1$ be a line in~$\mathbb{P}^2$
  and denote by~$\iota\colon L\hookrightarrow \mathbb{P}^2$ the closed immersion.
  We call the noncommutative space~$(L,\order{B})$ a \emph{Clifford conic order},
  where~$\order{B}\cong\order{A}|_L$.
\end{definition}

We denote by $\Delta_L = \iota^{-1}(\Delta)$ the pullback of the ramification divisor of~$\order{A}$.
If~$L$ is not contained in~$\Delta$,
it follows from \cref{proposition:restriction-order}
that~$\Delta_L$ is the ramification divisor of the~$\mathcal{O}_{L}$-order~$\order{B}$.

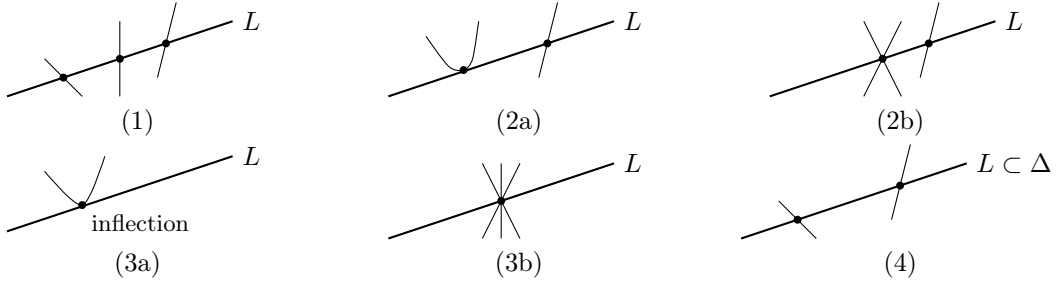
\begin{figure}[h]
  \centering
  \begin{subfigure}[b]{.3\textwidth}
    \centering
    \begin{tikzpicture}
      \draw[name path=L, thick] (0,0) -- (3,1) node [right] {$L$};
      \draw[name path=Delta1] (0.5,0.5) -- (1,0);
      \draw[name path=Delta2] (1.5,1) -- (1.5,0);
      \draw[name path=Delta3] (2.25,1.25) -- (2.,.25);

      \fill[name intersections={of=L and Delta1, by=a}] (a) circle (1.5pt);
      \fill[name intersections={of=L and Delta2, by=b}] (b) circle (1.5pt);
      \fill[name intersections={of=L and Delta3, by=c}] (c) circle (1.5pt);
    \end{tikzpicture}

    (1)
  \end{subfigure}
  \begin{subfigure}[b]{.3\textwidth}
    \centering
    \begin{tikzpicture}
      \draw[name path=L, thick] (0,0) -- (3,1) node [right] {$L$};
      \draw plot [smooth] coordinates {(0.5,0.8) (0.8,0.4) (1,0.35) (1.12,0.5) (1.2,1)};
      \draw[name path=Delta3] (2.25,1.25) -- (2.,.25);

      \fill (1,0.35) circle (1.5pt);
      \fill[name intersections={of=L and Delta3, by=c}] (c) circle (1.5pt);
    \end{tikzpicture}

    (2a)
  \end{subfigure}
  \begin{subfigure}[b]{.3\textwidth}
    \centering
    \begin{tikzpicture}
      \draw[name path=L, thick] (0,0) -- (3,1) node [right] {$L$};
      \draw[name path=Delta1] (1.25,0) -- (1.75,1);
      \draw[name path=Delta2] (1.25,1) -- (1.75,0);
      \draw[name path=Delta3] (2.25,1.25) -- (2.,.25);

      \fill[name intersections={of=L and Delta1, by=a}] (a) circle (1.5pt);
      \fill[name intersections={of=L and Delta3, by=c}] (c) circle (1.5pt);
    \end{tikzpicture}

    (2b)
  \end{subfigure}

  \begin{subfigure}[b]{.3\textwidth}
    \centering
    \begin{tikzpicture}
      \draw[name path=L, thick] (0,0) -- (3,1) node [right] {$L$};
      \draw plot [smooth] coordinates {(0.5,0.8) (1,0.35) (1.3,1)};

      \fill (1,0.35) circle (1.5pt) node [below right] {\small inflection};
    \end{tikzpicture}

    (3a)
  \end{subfigure}
  \begin{subfigure}[b]{.3\textwidth}
    \centering
    \begin{tikzpicture}
      \draw[name path=L, thick] (0,0) -- (3,1) node [right] {$L$};
      \draw[name path=Delta1] (1.25,0) -- (1.75,1);
      \draw[name path=Delta2] (1.25,1) -- (1.75,0);
      \draw[name path=Delta3] (1.5,0) -- (1.5,1);

      \fill[name intersections={of=L and Delta1, by=a}] (a) circle (1.5pt);
    \end{tikzpicture}

    (3b)
  \end{subfigure}
  \begin{subfigure}[b]{.3\textwidth}
    \centering
    \begin{tikzpicture}
      \draw[name path=L, thick] (0,0) -- (3,1) node [right] {$L\subset\Delta$};

      \draw[name path=Delta1] (0.5,0.5) -- (1,0);
      \draw[name path=Delta3] (2.25,1.25) -- (2.,.25);

      \fill[name intersections={of=L and Delta1, by=a}] (a) circle (1.5pt);
      \fill[name intersections={of=L and Delta3, by=c}] (c) circle (1.5pt);
    \end{tikzpicture}

    (4)
  \end{subfigure}

  \caption{The 6 possible intersections $L\cap\Delta$}
  \label{figure:intersection-types}
\end{figure}

Clifford conics can be classified by inspecting the geometry of the ramification divisor~$\Delta$
of the~$\mathcal{O}_{\mathbb{P}^2}$-order
and using the complete local model of terminal orders.
The first classification theorem describes Clifford conics in the language of sheaves of orders.
The~6~possible types will be referred to as~(1), (2a), (2b), (3a), (3b), and (4),
and a pictorial representation is given in \cref{figure:intersection-types}.

For a point~$p\in L$, we denote by~$\mathfrak{m}_p\unlhd \mathcal{O}_{L,p}$ the maximal ideal
and by $t$ a uniformizer of~$\mathfrak{m}_p$.

\begin{theorem}
  \label{theorem:clifford-conic-orders-classification}
  Let~$A$ be a graded Clifford algebra,
  let~$f$ define a Clifford conic.
  The Clifford conic order~$(L,\order{B})$ in~$(\mathbb{P}^2,\order{A})$
  is one of the following isomorphism types:
  \begin{enumerate}
    \item[(1)]
      If~$\# \Delta_L = 3$,
      then~$\order{B}$ is a hereditary~$\mathcal{O}_{L}$-order
      ramified over three points
      such that for every $p\in \Delta_L$
      \begin{equation}
        \label{eq:normalformquaternionhereditary}
        \order{B}\otimes_{\mathcal{O}_{L}} \widehat{\mathcal{O}}_{L,p}
        \cong
        (1,t)_2^{\widehat{\mathcal{O}}_{L,p}}.
      \end{equation}

    \item[(2)]
      If~$\#\Delta_L = 2$,
      then~$\order{B}$ is a Bass~$\mathcal{O}_L$-order
      ramified over two points.
      We distinguish the following two cases:
      \begin{enumerate}
        \item[(2a)]
          \emph{Tangent case.}
          Assume that~$L\cap \Delta^{\sing}=\emptyset$,
          then over a ramified point $p\in \Delta_L$ one has
          \begin{equation}
            \order{B}\otimes_{\mathcal{O}_{L}}\widehat{\mathcal{O}}_{L,p}
            \cong
            \begin{cases}
              (1,t^2)_2^{\widehat{\mathcal{O}}_{L,p}} & \text{if }\mult_p(\Delta,L) = 2,\\
              (1,t)_2^{\widehat{\mathcal{O}}_{L,p}} & \text{if }\mult_p(\Delta,L) = 1.
            \end{cases}
          \end{equation}
        \item[(2b)]
          \emph{Singular case.}
          Assume that~$L\cap \Delta^{\sing}$ is non-empty,
          then over a ramified point $p\in \Delta_L$ one has
          \begin{equation}
            \order{B}\otimes_{\mathcal{O}_{L}}\widehat{\mathcal{O}}_{L,p}
            \cong
            \begin{cases}
              (t,t)_2^{\widehat{\mathcal{O}}_{L,p}} & \text{if }p\in \Delta^{\sing},\\
              (1,t)_2^{\widehat{\mathcal{O}}_{L,p}} & \text{if }p\notin \Delta^{\sing}.
            \end{cases}
          \end{equation}
      \end{enumerate}

    \item[(3)]
      If~$\# \Delta_L = 1$,
      then~$\order{B}$ is a Bass~$\mathcal{O}_L$-order
      ramified over a single point~$p\in L$.
      We distinguish the following two cases:
      \begin{enumerate}
        \item[(3a)]
          \emph{Tangent case.}
          Assume that~$L\cap \Delta^{\sing}=\emptyset$,
          then~$\mathcal{O}_L$-order~$\order{B}$ restricts over $p$ to
          \begin{equation}
            \order{B}\otimes_{\mathcal{O}_{L}}\widehat{\mathcal{O}}_{L,p}
            \cong
            (1,t^3)_2^{\widehat{\mathcal{O}}_{L,p}}.
          \end{equation}
        \item[(3b)]
          \emph{Singular case.}
          Assume that~$L\cap\Delta^{\sing} = \Delta_L$,
          then~$\mathcal{O}_L$-order~$\order{B}$ restricts over $p$ to
          \begin{equation}
            \order{B}\otimes_{\mathcal{O}_{L}}\widehat{\mathcal{O}}_{L,p} \cong (t,t^2)_2^{\widehat{\mathcal{O}}_{L,p}}.
          \end{equation}
      \end{enumerate}

    \item[(4)]
      If~$L\subset \Delta$,
      then~$\order{B}$ is not an order.
      There exists an isomorphism~$\order{B}\cong \Cl^0(\mathcal{E},q,\mathcal{O}_{\bbP^1}(1))$
      with the even Clifford algebra of the vector bundle~$\mathcal{E} \cong \mathcal{O}_{L}^{\oplus 3}$
      and a quadratic form which is of rank $2$ everywhere except at two points, where it is of rank $1$.
      This corresponds to the pencil of conics with
      Segre symbol\footnote{
        See \cref{subsection:clifford-conics-comments} for more on Segre symbols.
        In the symbol~$[1, 1; \;; 1]$, the last entry encodes that the pencil is
        a cone over a pencil of length-2 subschemes, degenerating to a double point in two points of the base.
      }~$[1, 1; \;; 1]$.
  \end{enumerate}
\end{theorem}

In \cref{table:local-properties-clifford-conic-orders}
we summarize the results of this theorem
and some subsequent propositions.

The point scheme of the graded Clifford algebra
(or equivalently, the discriminant of the net of conics)
prescribes which isomorphism classes are feasible
for a given algebra:
not all types are possible for any given graded Clifford algebra.
This can again be read off from \cref{table:local-properties-clifford-conic-orders}.

We can in fact classify Clifford conics for all graded Clifford algebras simultaneously.
This will be explained in \cref{subsection:clifford-conics-comments}.

\begin{proof}
  By definition of a Clifford conic and \Cref{lemma:terminal},
  the $\mathcal{O}_L$-order~$\order{B}$ is the restriction of
  a terminal~$\mathcal{O}_{\mathbb{P}^2}$-order~$\order{A}$.
  We start with the complete local description of the algebra structure of~$\order{B}$.
  To do so, denote for~$p\in \mathbb{P}^2$ by~$\field[[x,y]]$
  the completion of the local ring~$\mathcal{O}_{\mathbb{P}^2,p}$ at~$p$.

  Assume first that~$L$ does not lie in the ramification divisor~$\Delta$.
  From \cite[Section 2.3]{MR2180454}, it follows that complete-locally
  terminal orders of rank~4 can only take the following two forms
  \begin{itemize}
    \item If $p\in \Delta^{\smooth}$, then
      \begin{equation}
        \label{equation:local-model-smooth}
        \widehat{\order{A}}_p
        =
        \order{A}\otimes_{\mathcal{O}_{\mathbb{P}^2,p}}\widehat{\mathcal{O}}_{\mathbb{P}^2,p}
        \cong
        \begin{pmatrix}
          \field[[x,y]]&\field[[x,y]]\\
          (x) & \field[[x,y]]
        \end{pmatrix}.
      \end{equation}
    \item If $p\in \Delta^{\sing}$, then
      \begin{equation}
        \label{equation:local-model-singular}
        \widehat{\order{A}}_p
        =
        \order{A}\otimes_{\mathcal{O}_{\mathbb{P}^2,p}}\widehat{\mathcal{O}}_{\mathbb{P}^2,p}
        \cong
        \frac{\field[[x,y]]\langle i,j\rangle}{(i^2-x, j^2-y, ij+ji)}.
      \end{equation}
  \end{itemize}

  Assume that~$p\in L\cap \Delta$,
  where we identify~$L$ with its image~$\iota(L)$.
  If $\mathcal{I}_L\unlhd \mathcal{O}_{\mathbb{P}^2,p}$
  is the ideal sheaf of~$L$,
  then
  \begin{equation}
    \order{B}\otimes_{\mathcal{O}_L}\widehat{\mathcal{O}}_{L,p}
    \cong
    \widehat{\order{A}}_p/\widehat{\mathcal{I}}_L\widehat{\order{A}}_p.
  \end{equation}
  Thus the description of~$\order{B}\otimes_{\mathcal{O}_L}\widehat{\mathcal{O}}_{L,p}$
  is a matter of distinguishing the possible intersections of
  a line with the ramification divisor~$\Delta$,
  which is of the form (A), (B), (D) or (E) as in \cref{theorem:van-gastel}.
  There are precisely six cases how~$\Delta$ and~$L$ can meet at~$p$,
  depicted in \cref{figure:intersection-types}.

  Since $\Delta$ is a curve of degree 3,
  there are three cases when~$p \in L\cap \Delta^{\smooth}$:
  the intersection is either transversal, tangential or an inflection point.
  After completion the ramification divisor is mapped to a line $\Delta = \VV(x)\subset\Spec\field[[x,y]]$
  in the model~\eqref{equation:local-model-smooth}.
  This in turn implies that~$L$ is a curve of degree 1, 2 or 3
  going through the origin transversally, tangentially or as an inflection point.

  If the intersection is transversal, the ideal sheaf of the curve,
  which we can assume to be $(y)$ as we work complete-locally,
  and $(x)$ generate the maximal ideal at $\field[[x,y]]$,
  and therefore we obtain
  \begin{equation}
    \label{equation:clifford-conic-hereditary}
    \order{B}\otimes_{\mathcal{O}_L}\widehat{\mathcal{O}}_{L,p}
    \cong (1,t)_2^{\widehat{\mathcal{O}}_{L,p}},
  \end{equation}
  which corresponds to case (1).
  This follows also from \Cref{proposition:restriction-order}.

  If the intersection is tangential (of degree~2 or~3),
  then after completion we can assume that~$L$ is mapped to a conic (resp.~cubic)
  of the form $\VV(x-y^i)$.
  We obtain the orders
  \begin{equation}
    \label{equation:clifford-conic-tiled}
    \order{B}\otimes_{\mathcal{O}_L}\widehat{\mathcal{O}}_{L,p}
    \cong
    \begin{pmatrix}
      \field[[y]] & \field[[y]] \\
      (y^{i}) & \field[[y]]
    \end{pmatrix}
    \cong
    (1,t^i)_2^{\widehat{\mathcal{O}}_{L,p}}.
  \end{equation}
  This describes the cases (2a) and (3a).

  In the situation $p\in L\cap \Delta^{\sing}$, there are two possibilities.
  The first we consider is when~$L$ intersects each of
  the branches of the ramification divisor transversally at~$p$.
  After completion we can assume~that $L = \VV(x-y)$.
  This leads to the case (2b)
  \begin{equation}
    \label{equation:clifford-conic-nodal}
    \order{B}\otimes_{\mathcal{O}_L}\widehat{\mathcal{O}}_{L,p}
    \cong
    (t,t)_2^{\widehat{\mathcal{O}}_{L,p}}.
  \end{equation}
  The last case is when $L$ is tangent to a branch of the ramification divisor
  (and consequently passes transversally through the other branch).
  The complete local picture is given by $L = \VV(x-y^2)$ and hence
  \begin{equation}\label{equation:clifford-conic-tangential}
    \order{B}\otimes_{\mathcal{O}_L}\widehat{\mathcal{O}}_{L,p}
    \cong
    (t,t^2)_2^{\widehat{\mathcal{O}}_{L,p}},
  \end{equation}
  which yields the case (3b).
  We show that the orders are Bass in \Cref{corollary:orders-gorenstein}.

  Finally, it is possible that~$L\subset\Delta$.
  In that case~$L$ intersects the other irreducible component(s) of~$\Delta$
  in~2 points.
  The description in the statement of the theorem follows readily
  from the description of the pencil of conics.
\end{proof}

The local information about Clifford conics can be used to
upgrade the commutative result that hypersurfaces in smooth varieties
are Gorenstein to a version for Clifford conics.
\begin{corollary}
  \label{corollary:orders-gorenstein}
  With the notation from \Cref{theorem:clifford-conic-orders-classification},
  assume that~$(L,\order{B})$ is a Clifford conic order such that~$\Delta_L$ is finite.
  Then~$\order{B}$ is a Bass~$\mathcal{O}_L$-order.
  In particular~$\order{B}$, is Gorenstein.
\end{corollary}

\begin{proof}
  By \cref{lemma:order-properties-local}
  one can check (complete) locally whether an order is Bass, resp.~Gorenstein.
  Therefore the result follows from the local description of the orders
  in \Cref{theorem:clifford-conic-orders-classification} and \Cref{lemma:clifford-conics-local} (iii).
\end{proof}
Note in particular, that by \Cref{lemma:clifford-conics-local} (ii)
none of the orders in the cases (2a), (2b), (3a) and (3b) are hereditary.

We collect the properties of the orders in the first three points in \cref{table:local-properties-clifford-conic-orders}.

\begin{table}[p]
  \centering
  \begin{tabular}{c|ccccccc}
    \toprule
    \textbf{type}            & \textbf{hereditary} & \textbf{Gorenstein} & \textbf{tiled} & \textbf{nodal} & \textbf{residually} & \textbf{appearance} \\
    \midrule
    (1)                      & yes                 & yes                 & yes            & yes            & split               & A, B, D, E \\
    \midrule
    (2a)                     & no                  & yes                 &                & yes            &                     & A, B, D \\
    $\mult_p(\Delta,L) = 2$  & no                  &                     & yes            & yes            & split               & \\
    $\mult_p(\Delta,L) = 1$  & yes                 &                     & yes            & yes            & split               & \\
    \addlinespace
    (2b)                     & no                  & yes                 &                & yes            &                     & B, D, E \\
    $p\in \Delta^{\sing}$    & no                  &                     & no             & yes            & ramified            & \\
    $p\notin \Delta^{\sing}$ & yes                 &                     & yes            & yes            & split               & \\
    \midrule
    (3a)                     & no                  & yes                 & yes            & no             & split               & A, B \\
    \addlinespace
    (3b)                     & no                  & yes                 & no             & no             & ramified            & B, D \\
    \midrule
    (4)                      & no                  & \multicolumn{4}{c}{\emph{not an order}}                                     & D, E \\
    \bottomrule
  \end{tabular}
  \caption{Summary of the (local) properties of the Clifford conics at $p\in \Delta_L$}
  \label{table:local-properties-clifford-conic-orders}
\end{table}

\subsection{Clifford conics as stacks}
\label{subsection:clifford-conics-stacks}
To the noncommutative space~$(\mathbb{P}^2,\order{A})$ obtained from a graded Clifford algebra
we can apply the bottom-up construction from \cite{2206.13359} as recalled in \cref{subsection:dictionary-dim-2}
and obtain
\begin{equation}
  \stack{S}_\can\xrightarrow{c} \stack{S}_\rt\xrightarrow{r}\mathbb{P}^2.
\end{equation}
This makes it possible to define the stacky incarnation of a Clifford conic order from \cref{definition:clifford-conic-order}.
\begin{definition}
  \label{definition:stacky-clifford-conic}
  Let~$(\mathbb{P}^2,\order{A})$ be the central Proj of a graded Clifford algebra,
  with associated stacky surface~$\stack{S}_\can$.
  Let $L\cong \mathbb{P}^1$ be a line in~$\mathbb{P}^2$.
  We call the Deligne--Mumford stack~$\stack{L}\colonequals L\times_{\mathbb{P}^2}\stack{S}_\can$
  a \emph{stacky Clifford conic}.
\end{definition}
The points with non-trivial stabilizer coincide with~$\Delta_L$.

The analogous statement to \Cref{theorem:clifford-conic-orders-classification}
becomes the following,
where we continue to use the labeling described in \cref{figure:intersection-types}.

\begin{theorem}
  \label{theorem:clifford-conic-stacks}
  Let~$A$ be a graded Clifford algebra,
  let~$f$ define a Clifford conic.
  The stacky Clifford conic~$\stack{L}$
  is one of the following isomorphism types:
  \begin{enumerate}
    \item[(1)]
      If~$\# \Delta_L = 3$,
      then for each point $p\in \Delta_L$
      \begin{equation}
        \stack{L}\times_{\mathbb{P}^2}\Spec(\widehat{\mathcal{O}}_{L,p})
        \cong
        \left[\Spec\left(\frac{\field[[t,u]]}{(u^2-t)}\right)/\mu_2\right],
      \end{equation}
      where~$\mu_2$ acts by $u\mapsto -u$.
      In particular, globally, one obtains the root stack $\stack{L} \cong \sqrt[2]{\mathbb{P}^1;\Delta_L}$.

    \item[(2)]
      If~$\#\Delta_L = 2$, we distinguish the following two cases:
      \begin{enumerate}
        \item[(2a)]
          \emph{Tangent case.}
          Assume that $L\cap\Delta^{\sing} = \emptyset$.
          Then over $\Delta_L$ we have
          \begin{equation}
            \stack{L}\times_{\mathbb{P}^2}\Spec(\widehat{\mathcal{O}}_{L,p})
            \cong
            \begin{cases}
              \left[\Spec\left(\frac{\field[[t,u]]}{(u^2-t^2)}\right)/\mu_2\right]&\text{if }\mult_p(\Delta,L) = 2,\\
              \left[\Spec\left(\frac{\field[[t,u]]}{(u^2-t)}\right)/\mu_2\right]&\text{if }\mult_p(\Delta,L) = 1.
            \end{cases}
          \end{equation}
          In both cases, $\mu_2$ acts by $u\mapsto -u$.

          In particular, globally, one obtains the root stack $\stack{L} \cong \sqrt[2]{\mathbb{P}^1;2[q]+[p]}$
          at a non-reduced divisor, where $\mult_q(\Delta,L) = 2$.

        \item[(2b)]
          \emph{Singular case.}
          Assume that $L\cap \Delta^{\sing}$ is non-empty.
          Then over $\Delta_L$ we have
          \begin{equation}
            \begin{aligned}
              \left[\Spec\left(\frac{\field[[v,w]]}{(w^2-v^2)}\right)/\mu_2\right]\to \stack{L}\times_{\mathbb{P}^2}\Spec(\widehat{\mathcal{O}}_{L,p})\quad & \text{if } p\in \Delta^{\sing}, \\
              \stack{L}\times_{\mathbb{P}^2}\Spec(\widehat{\mathcal{O}}_{L,p}) \cong \left[\Spec\left(\frac{\field[[t,u]]}{(u^2-t)}\right)/\mu_2\right]\quad & \text{if }p\notin \Delta^{\sing}.
            \end{aligned}
          \end{equation}
          In the first case, $\mu_2$ acts by $v,w\mapsto -v,-w$ and the map is an \'etale covering.
          In the second case $\mu_2$ acts by $u\mapsto -u$.
      \end{enumerate}

    \item[(3)]
      If~$\#\Delta_L=1$, then one distinguishes again two situations:
      \begin{enumerate}
        \item[(3a)]
          \emph{Tangent case.}
          If~$L\cap\Delta^{\sing} =\emptyset$,
          then at the (unique) closed point $p\in \Delta_L$ we have
          \begin{equation}
            \stack{L}\times_{\mathbb{P}^2}\Spec(\widehat{\mathcal{O}}_{L,p})\cong \left[\Spec\left(\frac{\field[[t,u]]}{(u^2-t^3)}\right)/\mu_2\right],
          \end{equation}
          where~$\mu_2$ acts by~$u\mapsto -u$.
          In particular, one obtains a root stack~$\stack{L} \cong \sqrt[2]{\mathbb{P}^1;3\Delta_L}$ globally.
        \item[(3b)]
          \emph{Singular case.}
          If~$L\cap\Delta^{\sing} =\Delta_L$,
          then at the (unique) closed point~$p\in\Delta_L$
          one obtains an \'etale covering
          \begin{equation}
            \left[\Spec\left(\frac{\field[[v,w]]}{(w^4-v^2)}\right)/\mu_2\right]\to \stack{L}\times_{\mathbb{P}^2}\Spec(\widehat{\mathcal{O}}_{L,p}),
          \end{equation}
          where $\mu_2$ acts by $v,w\mapsto -v,w$.
      \end{enumerate}

    \item[(4)]
      If~$L\subset \Delta_L$,
      then the generic stabilizer is~$\mu_2$.
      If~$p\notin\Delta^{\sing}$, then
      we have
      \begin{equation}
        \stack{L}\times_{\mathbb{P}^2}\Spec(\widehat{\mathcal{O}}_{L,p}) \cong \left[\Spec\left(\frac{\field[[t,u]]}{(u^2)}\right)/\mu_2\right]
      \end{equation}
      whereas if~$p\in\Delta^{\sing}$,
      we have an \'etale double cover
      \begin{equation}
        \left[\Spec\left(\frac{\field[[v,w]]}{(v^2)}\right)/\mu_2\right]\to \stack{L}\times_{\mathbb{P}^2}\Spec(\widehat{\mathcal{O}}_{L,p}).
      \end{equation}
  \end{enumerate}
\end{theorem}

\begin{proof}
  Since $\Cl(M)$ defines a terminal order on $\mathbb{P}^2$,
  as in the proof of \Cref{theorem:clifford-conic-orders-classification},
  there are two different possibilities for the local structure of the order
  at a point $p\in \Delta\cap L$.
  We will write~$\field[[x,y]]=\widehat{\mathcal{O}}_{\mathbb{P}^2,p}$.

  We have the following distinction for the canonical stack constructed from the terminal order:
  \begin{itemize}
    \item In the situation \eqref{equation:local-model-smooth},
      one obtains
      \begin{equation}
        \label{equation:stacky-model-smooth}
        \stack{S}_\can\times_{\mathbb{P}^2}\Spec(\widehat{\mathcal{O}}_{\mathbb{P}^2,p})
        \cong
        \stack{S}_\rt\times_{\mathbb{P}^2}\Spec(\widehat{\mathcal{O}}_{\mathbb{P}^2,p})
        \cong
        \left[\Spec\left(\frac{\field[[x,y,u]]}{(u^2-x)}\right)/\mu_2\right].
      \end{equation}
    \item In the situation \eqref{equation:local-model-singular},
      one obtains an \'etale neighborhood of the canonical stack
    \begin{equation}\label{equation:stacky-model-singular}
       \left[\Spec\left(\field[[v,w]]\right)/\mu_2\right]\to \stack{S}_\can\times_{\mathbb{P}^2}\Spec(\widehat{\mathcal{O}}_{\mathbb{P}^2,p})
    \end{equation}
    over the point $p$, where the atlas of the associated root stack
    \begin{equation}
      \Spec\left(\frac{\field[[x,y,u]]}{(u^2-xy)}\right)\to \stack{S}_\rt\times_{\mathbb{P}^2}\Spec(\widehat{\mathcal{O}}_{\mathbb{P}^2,p}) \cong \left[\Spec\left(\frac{\field[[x,y,u]]}{(u^2-xy)}\right)/\mu_2\right]
    \end{equation}
    is the invariant ring of $\field[[v,w]]$ under the~$\mu_2$-action $v\mapsto -v$ and $u\mapsto -u$.
  \end{itemize}

  The case-by-case analysis is now similar to the proof of \Cref{theorem:clifford-conic-orders-classification}.
  If $p\in \Delta^{\smooth}$, we are in the case \eqref{equation:stacky-model-smooth}
  and the line $L$ becomes after a local automorphism at~$p$ a curve of at most degree~3.
  This explains the cases (1), (2a), (3a).

  If $p\in L\cap \Delta^{\sing}$,
  one has a local situation as in \eqref{equation:stacky-model-singular}
  and $L$ maps to a curve of at most degree 2 in the local neighborhood.
  If $L$ remains a line intersecting both irreducible components~$\VV(x)$ and~$\VV(y)$ of~$\Delta$ transversally,
  it is mapped to~$\VV(w^2-v^2)$ (after a suitable choice of coordinates).
  This leads to the case (2b).

  Otherwise, we can assume that~$L = \VV(x-y^2)$,
  which maps to~$\VV(v^2-w^4)\subset \Spec\field[[v,w]]$.
  This describes the case (3b).

  If~$L$ lies in the ramification divisor,
  say~$L = V(x)$ locally,
  then generically one considers the setup \eqref{equation:stacky-model-smooth}
  and we obtain
  \begin{equation}
    \stack{L}\times_{\mathbb{P}^2}\Spec(\widehat{\mathcal{O}}_{L,p})
    \cong
    \left[\Spec\left(\frac{\field[[y,u]]}{(u^2)}\right)/\mu_2\right].
  \end{equation}
  For the two points $p\in \Delta^{\sing}\cap L$,
  note that $x\mapsto v^2$.
  Therefore, using \eqref{equation:stacky-model-singular},
  we obtain the following description over these points
  \begin{equation}
    \left[\Spec\left(\frac{\field[[v,w]]}{(v^2)}\right)/\mu_2\right]
    \to
    \stack{L}\times_{\mathbb{P}^2}\Spec(\widehat{\mathcal{O}}_{L,p}).
  \end{equation}
\end{proof}

We have also summarized these results in \cref{table:properties-stacky-clifford-conics}.

\begin{remark}
  From \cref{lemma:clifford-conics-local} (v) it follows that
  the cases (2a) and (2b) of \cref{theorem:clifford-conic-orders-classification}
  are nodal orders which are not hereditary.

  In the case (2a) the inclusion into the overorder
  is the hereditary $\mathcal{O}_L$-order ramified in one point.
  This inclusion translates to a map of root stacks $\sqrt[2]{\bbP^1;2q+p}\to \sqrt[2]{\bbP^1;p}$.
  The vanishing of the stackiness at~$q$ follows from the fact that $\rad (1,t^2)$ is,
  up to conjugation by an element of $\Mat_2(\field((t)))$,
  the radical of the maximal order $\Mat_2(R)$, see e.g.~\cite[\S 23.4.15]{MR4279905}.

  In the case (2b) the nodal order is mapped into a hereditary order ramified over the same two points.
  This inclusion corresponds to the map $\stack{L}\to \sqrt[2]{\bbP^1;p+q}$
  from the canonical stack to the root stack in \cref{theorem:clifford-conic-stacks}.
\end{remark}

\begin{table}[p]
  \centering
  \begin{tabular}{c|ccccccc}
    \toprule
    \textbf{type}            & \textbf{smooth} & \textbf{generically scheme} & \textbf{root stack} & \textbf{appearance} \\
    \midrule
    (1)                      & yes             & yes                         & yes                 & A, B, D, E \\
    \midrule
    (2a)                     & no              & yes                         & yes                 & A, B, D \\
    $\mult_p(\Delta,L) = 2$  & no              & \\
    $\mult_p(\Delta,L) = 1$  & yes             & \\
    \addlinespace
    (2b)                     & no              & yes                         & no                  & B, D, E \\
    $p\in \Delta^{\sing}$    & no              &                             & no \\
    $p\notin \Delta^{\sing}$ & yes             &                             & yes \\
    \midrule
    (3a)                     & no              & yes                         & yes                 & A, B \\
    \addlinespace
    (3b)                     & no              & yes                         & no                  & B, D \\
    \midrule
    (4)                      & no              & no                          & no                  & D, E \\
    \bottomrule
  \end{tabular}
  \caption{Summary of the properties of the stacky Clifford conics~$\stack{L}$}
  \label{table:properties-stacky-clifford-conics}
\end{table}

\subsection{Classification and comments}
\label{subsection:clifford-conics-comments}
We end our discussion of noncommutative conics,
by pointing out how our results relate to those in \cite{MR4794245,MR4531545},
and how they give more information about the derived categories of noncommutative conics.

\paragraph{The global classification of Clifford conics}
The classifications in \cref{theorem:clifford-conic-orders-classification,theorem:clifford-conic-stacks}
take place relative to a fixed graded Clifford algebra~$A$.
However, \cite[Theorem~5.11]{MR4531545}
gives a classification of noncommutative conics into~6 isomorphism types
(where isomorphism refers to equivalence of the qgr),
taking all possible ambient algebras~$A$ into account.

The following proposition explains how to upgrade the algebra-per-algebra classification
from \cref{theorem:clifford-conic-orders-classification,theorem:clifford-conic-stacks}
to a global classification.
Observe that,
if~$f$ is a Clifford conic inside the graded Clifford algebra~$\Cl(M)$,
where~$M$ is a net of conics,
the restriction to~$L\subset\mathbb{P}^2$ gives a pencil of conics.
\begin{theorem}
  \label{theorem:classification}
  Let~$A,A'$ be two graded Clifford algebras,
  and let~$f,f'$ define two Clifford conics.
  The following are equivalent:
  \begin{enumerate}
    \item\label{enumerate:classification-qgr}
      $\qgr A/(f)$ and~$\qgr A'/(f')$ are isomorphic as noncommutative projective varieties.
    \item\label{enumerate:classification-segre}
      the Segre symbols of the pencils of conics attached to~$f$ and~$f'$ are equal.
    \item\label{enumerate:classification-type}
      they have the same type from the set~$\{\normalfont(1),(2a),(2b),(3a),(3b),4\}$,
  \end{enumerate}
\end{theorem}
Segre symbols are a means of classifying pencils of quadric hypersurfaces,
and thus in particular conics.
We refer to \cite[Chapter~XIII]{MR1288306} for their definition,
or \cite{segre} for a survey by the second author.

\begin{proof}
  The equivalence of \cref{enumerate:classification-qgr} and \cref{enumerate:classification-segre}
  follows from \cite[Theorem~3.28 and Remark~3.30]{MR4794245}.
  Note that it must be pointed out that not all possible Segre symbols arise:
  by assumption the net of conics (and thus the pencil of conics) is basepoint-free,
  which excludes two types.

  The equivalence  of \cref{enumerate:classification-segre} and \cref{enumerate:classification-type}
  follows from the fact that the Segre symbol determines the set of singular fibers
  of the pencil of conics,
  and vice versa.
  And the set of singular fibers describes how often~$L$ intersects~$\Delta$
  (transversely or tangentially) at a smooth point,
  thus giving two lines as the fiber,
  and how often~$L$ intersects~$\Delta$ at a singular point,
  thus giving a double line as the fiber.
\end{proof}

\begin{remark}
  \label{remark:stacky-classification}
  The proof of \cref{theorem:classification} is not independent of the classification results in \cite{MR4531545,MR4794245}.
  However, the stacky description of the noncommutative conics from \cref{theorem:clifford-conic-stacks}
  makes it clear that,
  by the 3-transitivity of the~$\mathrm{PGL}_2$-action on~$\mathbb{P}^1$,
  the root stacks of type~(1), (2a), resp.~(3a)
  are all isomorphic to one another within a given type.
  For the types (2b), (3b), and (4),
  the absence of a bottom-up description for singular stacks
  prevents us from giving a rigorous proof of \cref{theorem:classification}
  using the stacky interpretation.
\end{remark}

In \cref{table:comparison-table} we compare our point-of-view with the classification in \cite{MR4531545,MR4794245}.
The point scheme~$E_B$ refers to \cite[Table 4]{MR4531545},
where it is denoted~$E_A$,
and it is determined in Example~4.6 and Theorem~4.13 of op.~cit.

\begin{table}
  \centering
  \begin{tabular}{cccc}
    \toprule
    type     & $\mathbf{\# E_B}$ & $\mathrm{C}(B)$                 & Segre symbol \\
    \midrule
    (1)      & $6$               & $\field^4$                      & $[1,1,1]$ \\
    \midrule
    (2a)     & $4$               & $\field[t]/(t^2)\times\field^2$ & $[2,1]$ \\
    (2b)     & $3$               & $(\field[t]/(t^2))^2$           & $[(1,1),1]$ \\
    \midrule
    (3a)     & $2$               & $\field[t]/(t^3)\times\field$   & $[3]$ \\
    (3b)     & $1$               & $\field[t]/(t^4)$               & $[(2, 1)]$ \\
    \midrule
    (4)      & $\infty$          & $\field[s,t]/(s^2,t^2)$         & $[1,1;;1]$ \\
    \bottomrule
  \end{tabular}
  \caption{Dictionary between our types and the classification in \cite[Table 4]{MR4531545} and \cite[Table~3]{MR4794245}}
  \label{table:comparison-table}
\end{table}

\paragraph{Derived categories of noncommutative conics}
As a part of their classification \cite[Theorem 5.10]{MR4531545},
Hu--Matsuno--Mori obtain a description of the derived category of a noncommutative conic
in terms of Orlov's semiorthogonal decomposition \cite[Theorem 16]{MR2641200}
\begin{equation}
  \label{equation:orlov-qgr}
  \derived^\bounded(\qgr B) = \langle \order{B}, \derived^\bounded(\mathrm{C}(B))\rangle,
\end{equation}
using the result by Smith--Van den Bergh \cite[Theorem 5.2]{MR3108697}
that the singularity category is actually given by
the bounded derived category of modules over the finite-dimensional~$\field$-algebra~$\mathrm{C}(B)$.
Here~$B\colonequals A/(f)$ is the homogeneous coordinate ring of a noncommutative conic as before.
The algebras~$\mathrm{C}(B)$ for the six cases of noncommutative conics can be found in \cite[Table 4]{MR4531545},
as recalled in \cref{table:comparison-table}.

However, from our perspective
we have another natural semiorthogonal decomposition,
which (assuming the vanishing of the Brauer class)
in 3 out of 5 cases gives a more explicit description of the derived category
because the gluing data can be explicitly read off.

Namely,
let~$\pi\colon\stack{L}\to L$ be the coarse moduli space for the stacky curve attached to the noncommutative conic.
Then~$\mathbf{R}\pi_*\mathcal{O}_{\stack{L}}\cong\mathcal{O}_L$,
so that by the projection formula
we can obtain the semiorthogonal decomposition
\begin{equation}
  \label{equation:root-stack-sod}
  \derived^\bounded(\coh\stack{L})
  =
  \langle
    \pi^*\derived^\bounded(\mathbb{P}^1),
    \mathcal{T}
  \rangle,
\end{equation}
where~$\mathcal{T}$ is defined as the semiorthogonal complement.
Because~$\stack{L}$ is a root stack for types~(1), (2a) and~(3a),
the semiorthogonal decomposition for root stacks from \cite[Theorem 4.7]{MR3573964}
gives us
\begin{equation}
  \mathcal{T}
  \simeq
  \begin{cases}
    \derived^\bounded(k^3) & \text{\normalfont type (1)} \\
    \derived^\bounded(k\times k[t]/(t^2)) & \text{\normalfont type (2a)} \\
    \derived^\bounded(k[t]/(t^3)) & \text{\normalfont type (3a)}
  \end{cases}
\end{equation}
The sheaf of algebras~$\order{A}|_L$ corresponds to an Azumaya algebra on the stack~$\stack{L}$.
For type~(1) we have that it is split, as in \cref{corollary:no-azumaya}.
We expect that some version of the vanishing of the Brauer group of a curve as in \cite[Proposition~7.3.2]{MR4304038}
holds for the stacky curves appearing in \cref{theorem:restriction},
at least when there is no generic stabilizer.
This would give rise to an equivalence
\begin{equation}
  \derived^\bounded(\qgr B)
  \simeq
  \derived^\bounded(\coh\stack{L}),
\end{equation}
and thus \eqref{equation:root-stack-sod}
really is an alternative to \eqref{equation:orlov-qgr}.

\begin{remark}
  \label{remark:comparison-sods}
  Even in the nicest possible case of type (1),
  the semiorthogonal decompositions \eqref{equation:orlov-qgr}
  and \eqref{equation:root-stack-sod} are different.
  The former gives the path algebra of the extended Dynkin quiver~$\widetilde{\mathrm{D}}_4$,
  \begin{equation}
    \widetilde{\mathrm{D}}_4:\hspace{1ex}
    \begin{tikzpicture}[baseline={([yshift=-2pt]current bounding box.center)},vertex/.style={draw, circle, inner sep=0pt, outer sep=1pt, text width=2mm}]
      \node[vertex] (a) at (0:.8cm)   {};
      \node[vertex] (b) at (90:.8cm)  {};
      \node[vertex] (c) at (180:.8cm) {};
      \node[vertex] (d) at (270:.8cm) {};
      \node[vertex] (o) at (0,0) {};
      \draw[-{Classical TikZ Rightarrow[]}] (o) -- (a);
      \draw[-{Classical TikZ Rightarrow[]}] (o) -- (b);
      \draw[-{Classical TikZ Rightarrow[]}] (o) -- (c);
      \draw[-{Classical TikZ Rightarrow[]}] (o) -- (d);
    \end{tikzpicture}
  \end{equation}
  whereas the latter is the quotient of the path algebra for the squid quiver
  \begin{equation}
    \begin{tikzpicture}[baseline,vertex/.style={draw, circle, inner sep=0pt, outer sep=1pt, text width=2mm}]
      \node[vertex] (a) at (1,0) {};
      \node[vertex] (b) at (2,0) {};
      \node[vertex] (c) at (3,.75) {};
      \node[vertex] (d) at (3,0) {};
      \node[vertex] (e) at (3,-0.75) {};
      \draw[-{Classical TikZ Rightarrow[]}] (a) to [bend left] node [above] {$\scriptstyle x$} (b);
      \draw[-{Classical TikZ Rightarrow[]}] (a) to [bend right] node [below] {$\scriptstyle y$} (b);
      \draw[-{Classical TikZ Rightarrow[]}] (b) -- node [above] {$\scriptstyle \alpha$} (c);
      \draw[-{Classical TikZ Rightarrow[]}] (b) -- node [above, pos=.7,yshift=-2pt] {$\scriptstyle \beta$} (d);
      \draw[-{Classical TikZ Rightarrow[]}] (b) -- node [below] {$\scriptstyle \gamma$} (e);
    \end{tikzpicture}
  \end{equation}
  by the ideal of relations~$(\alpha x,\beta y,\gamma x-\gamma y)$.
  It is a classical fact that these are derived equivalent.
\end{remark}

\section{Noncommutative skew cubics}
\label{section:noncommutative-cubics}

Going beyond the case of noncommutative conics, we consider noncommutative plane cubics.
The noncommutative projective planes we will consider
are given by skew polynomial algebras~$A=\field_q[x,y,z]$ with coefficient matrix
\begin{equation}
  \label{equation:coefficient-matrix}
  q=\begin{pmatrix}
    1 & q_{1,2} & q_{1,3} \\
    q_{1,2}^{-1} & 1 & q_{2,3} \\
    q_{1,3}^{-1} & q_{2,3}^{-1} & 1
  \end{pmatrix}
\end{equation}
where~$q_{i,j} \in \field^\times$ are cube roots of unity.
This guarantees the presence of a natural net of central cubics.
The algebra~$A$ is always finite over its center,
which is described explicitly in \cite{MR1624000},
and the central Proj is always given by a maximal order~$\order{A}$ on~$\mathbb{P}^2$.
See \cref{subsection:noncommutative-via-central-curves} for more details.

There are two distinct cases for~$(\mathbb{P}^2,\order{A})$ which depend on whether the condition
\begin{equation}
  \label{equation:kanazawa-independence-restated}
  q_{1,2}^{-1}q_{1,3}^{-1}=q_{1,2}q_{2,3}^{-1}=q_{1,3}q_{2,3}
\end{equation}
is satisfied or not.
As in the conic setting of \Cref{section:noncommutative-conics},
we can therefore study noncommutative curves by considering quotients by an element in the net
\begin{equation}
  f \in \ZZ(A)_3 = \langle x^3, y^3, z^3\rangle.
\end{equation}
In \cite{MR3313507},
the quotient by the central element~$f= x^3+y^3+z^3\in\ZZ(A)_3$
is considered\footnote{Op.~cit.~also considers higher-dimensional versions of this setup.}
to describe a noncommutative ``Fermat'' cubic~$\qgr A/(f)$.
It is shown in op.~cit.~that this quotient is Calabi--Yau
if and only if
the condition \eqref{equation:kanazawa-independence-restated} holds.
We will generalise these results to arbitrary noncommutative cubics,
and independent of whether condition \eqref{equation:kanazawa-independence-restated} holds.

\begin{proposition}
  \label{proposition:kanazawa-fermat-cubic-1}
  Let~$A=\field_q[x,y,z]$ be a skew polynomial algebra
  where the~$q_{i,j}$ are cube roots of unity satisfying \eqref{equation:kanazawa-independence-restated}.
  For every~$f = a x^3+ b y^3+ cz^3$ defined by some $[a:b:c]\in \mathbb{P}^2$,
  and writing~$B\colonequals A/(f)$,
  we have
  \begin{equation}
    \qgr B\simeq\coh E
  \end{equation}
  where~$E\subset \mathbb{P}_{[x:y:z]}^2$ is the cubic curve with equation $f$.
\end{proposition}

\Cref{proposition:kanazawa-fermat-cubic-1}
is an upgrade of \cite[Theorem~1.1]{MR3313507}:
by the classification of 1-Calabi--Yau categories in \cite{MR2427460}
one knows that~$\qgr B$ for~$f=x^3+y^3+z^3$
must be~$\coh E$ for \emph{some} elliptic curve~$E$.
In \cref{subsection:kanazawa-is-fermat}
we prove that it is the Fermat elliptic curve,
and moreover extend the description to \emph{all} elements in the net~$\langle x^3,y^3,z^3\rangle$ of central cubics.

Note that this net includes singular fibres: the curve~$E$ is singular when $abc = 0$.
If only one of the~$a,b,c$ vanishes,
then~$E$ are three concurrent lines,
if two of the~$a,b,c$ vanish,
then~$E$ is the non-reduced triple line.

We also consider the case where the condition of \cite{MR3313507} does not hold, which is new.
In \cref{subsection:non-kanazawa-is-tubular} we prove the following result,
giving a complete description of the elements in the net of noncommutative cubics.

\begin{proposition}
  \label{proposition:kanazawa-fermat-cubic-2}
  Let~$A=\field_q[x,y,z]$ be a skew polynomial algebra
  where the~$q_{i,j}$ are cube roots of unity such that \eqref{equation:kanazawa-independence-restated} does not hold.
  For every~$f = a x^3+ b y^3+ cz^3$ defined by some $[a:b:c]\in \mathbb{P}^2$,
  and writing~$B\colonequals A/(f)$,
  we have
  \begin{equation}
    \qgr B \simeq \coh(\stack{C},\order{B}),
  \end{equation}
  where $\stack{C}$ is a separated Deligne--Mumford stack of dimension $1$ with coarse moduli space $\mathbb{P}^1$,
  and~$\order{B}$ is an Azumaya algebra on~$\stack{C}$,
  with the following description:
  \begin{enumerate}
    \item if $abc\neq 0$,
      then $\stack{C}$ is isomorphic to
      the (smooth) root stack $\stack{C} \cong \sqrt[3]{\kern1pt\mathbb{P}^1\kern-1pt;0+1+\infty}$
      and~$\order{B}$ is split;
    \item if $abc=0$ and exactly one of $a,b,c$ is $0$,
      then~$\stack{C}$ has trivial generic stabiliser and two stacky points:
      it is the union in the Zariski topology
      of the (smooth) root stack~$\sqrt[3]{\mathbb{A}^1;0}$
      and a singular quotient stack;
    \item if exactly two of $a,b,c$ are $0$,
      then~$\stack{C}$ has generic stabiliser $\mu_3$
      and two points with non-generic behavior.
  \end{enumerate}
\end{proposition}

Hence, we now obtain a family of \emph{stacks} over $\mathbb{P}^1$, some of which are singular,
and isotrivial on the complement of~$abc=0$.
We give a full description of the singular fibres in the net in \cref{subsection:non-kanazawa-is-tubular}.
We expect that the Azumaya algebra~$\order{B}$ is split
also for singular fibres.

\subsection{Kanazawa's skew Fermat cubic is the Fermat cubic}
\label{subsection:kanazawa-is-fermat}
When the condition \eqref{equation:kanazawa-independence-restated} holds,
\cite[\S 4]{MR1624000} can be used to show the following easy lemma.
\begin{lemma}
  \label{lemma:skew-is-twisted}
  If~$A$ is the skew polynomial algebra with coefficient matrix as in \eqref{equation:coefficient-matrix}
  satisfying \eqref{equation:kanazawa-independence-restated}
  then~$A$ is a twisted polynomial ring,
  i.e.,
  there exists an isomorphism of algebras
  \begin{equation}
    \phi\colon A\to \field[X,Y,Z]^\theta,\quad x\mapsto X,\ y\mapsto Y,\ z\mapsto Z,
  \end{equation}
  where~$\field[X,Y,Z]^\theta = (\field[X,Y,Z],*)$ has the twisted product defined by~$f * g = f g^\theta$
  for the ring automorphism~$\theta \colon \field[X,Y,Z] \to \field[X,Y,Z]$
  with $X^\theta = X$, $Y^\theta = q_{1,2}Y$, and $Z^\theta = q_{1,3} Z$.
\end{lemma}

With this lemma in hand,
the proof of \cref{proposition:kanazawa-fermat-cubic-1} is easy.
\begin{proof}[Proof of \cref{proposition:kanazawa-fermat-cubic-1}]
  The twisted polynomial ring has center generated by elements of degree $3$,
  thus by \Cref{lemma:veronese-construction}
  the central Proj is constructed from the 3-Veronese $A^{(3)}$.
  It follows by \Cref{lemma:skew-is-twisted} that the 3-Veronese is given by
  \begin{equation}
    A^{(3)} \cong (\field[X,Y,Z]^\theta)^{(3)} = \field[X,Y,Z]^{(3)},
  \end{equation}
  where the final equality is \cite[Corollary 4.5]{MR1624000}.
  In particular the central Proj is simply the commutative plane
  \begin{equation}
    (\Proj \ZZ(A^{(3)}), \order{A}^{(3)}) \cong (\Proj \field[X,Y,Z]^{(3)},\mathcal{O}_{\Proj \field[X,Y,Z]^{(3)}}) \cong (\mathbb{P}^2,\mathcal{O}_{\mathbb{P}^2}).
  \end{equation}
  Under the algebra isomorphism from \Cref{lemma:skew-is-twisted} the central element~$f=ax^3+by^3+cz^3$ is mapped to
  \begin{equation}
    \phi(f) = aX*X*X + bY*Y*Y + cZ*Z*Z = aX^3 + bq_{1,2}^2 Y^3 + cq_{1,3}^2 Z^3.
  \end{equation}
  Multiplying $Y$ and $Z$ by an appropriate~$9$th root of unity we recover the equation $aX^3+bY^3+cZ^3$,
  which defines a cubic curve~$E\subset \mathbb{P}^2$.
  We conclude that
  \begin{equation}
    \qgr A/(f)
    \simeq \qgr \left(A^{(3)}/(f)\right)
    \simeq \qgr\left( \field[X,Y,Z]/(aX^3+bY^3+cZ^3) \right)
    \simeq \coh(E),
  \end{equation}
  as claimed.
\end{proof}

\begin{remark}
  \label{remark:kanazawa-caveat-CY-curve}
  \Cref{proposition:kanazawa-fermat-cubic-1}
  explains the need for a caveat to the statement of \cite[Proposition~3.7]{MR3313507},
  in that we have to exclude skew Fermat cubic curves from loc.~cit.,
  as we do not obtain a genuinely noncommutative Calabi--Yau curve.
  In fact, no such thing exists,
  by the classification in \cite{MR2427460}.
\end{remark}

\subsection{The other skew cubics as (degenerations of) tubular weighted projective lines}
\label{subsection:non-kanazawa-is-tubular}
Now suppose that $A = \field_q[x,y,z]$ is a skew polynomial algebra
for which the condition \eqref{equation:kanazawa-independence-restated} does not hold.
The central Proj of $A$ can be described using \cite[\S 3]{MR1624000} and is as follows.

\begin{lemma}
  \label{lem:non-kanazwa}
  If~$A = \field_q[x,y,z]$ with coefficient matrix as in \eqref{equation:coefficient-matrix}
  which does not satisfy \eqref{equation:kanazawa-independence-restated},
  then $\ZZ(A)$ has generators of coprime degree and the central Proj is given by
  \begin{equation}
    (\mathbb{P}^2,\order{A}) =
    (\Proj \ZZ(A)^{(3)}, \order{A}^{(3)}) = (\Proj \field[x^3,y^3,z^3]^{(3)}, \order{A}^{(3)}).
  \end{equation}
\end{lemma}
\begin{proof}
  By inspection the elements $x^3,y^3,z^3$ are generators of $\ZZ(A)$ of degree $3$,
  which implies that the greatest common denominator of the degrees of the generators is either $1$ or $3$.
  As in \cite[\S 3]{MR1624000} we take a primitive $3$rd root of unity $\zeta \in \field^*$ and present
  \begin{equation}
    q_{1,2} = \zeta^{-c},\quad
    q_{1,3} = \zeta^b,\quad
    q_{2,3} = \zeta^{-a},
  \end{equation}
  for some $1\leq a,b,c \leq 3$.
  One checks that the condition \eqref{equation:kanazawa-independence-restated} translates to $a+b+c \equiv 0 \bmod 3$.
  Because this condition fails by assumption,
  it follows that $a+b+c$ is not divisible by $3$,
  and it therefore follows by \cite[Corollary 3.4]{MR1624000}
  that the gcd of the generators of $\ZZ(A)$ is also not divisible by $3$.
  Hence $\ZZ(A)$ is generated in coprime degrees,
  and it follows from \Cref{lemma:veronese-construction} that the central Proj is given by
  \begin{equation}
    (\Proj \ZZ(A)^{(3)}, \order{A}^{(3)}) = (\Proj \field[x^3,y^3,z^3]^{(3)}, \order{A}^{(3)}),
  \end{equation}
  where the order $\widetilde{A}^{(3)}$ corresponds to $A^{(3)} \in \qgr \ZZ(A)^{(3)}$.
\end{proof}

To simplify the notation and avoid confusion, we write~$\ZZ(A)^{(3)} = \field[X,Y,Z]$
where $X$, $Y$, $Z$ are the degree-$1$ elements corresponding to $x^3$, $y^3$, $z^3$.

\begin{lemma}
  \label{lemma:nonkanazawa-ramification}
  The order ${\order{A}}^{(3)}$ has ramification index~$3$ on the three lines $\VV(XYZ) \subset \mathbb{P}^2$.
\end{lemma}

\begin{proof}
  It is straightforward to see that $A^{(3)}$ is generated as a module over $\field[X,Y,Z]$
  by monomials $x^iy^jz^k$ with $i,j,k < 3$ and $i+j+k$ divisible by $3$,
  which are the following $9$ elements:
  \begin{equation}
    1,\quad
    x^2y,\quad xy^2,\quad x^2z,\quad xz^2,\quad y^2z,\quad yz^2
    ,\quad xyz,\quad x^2y^2z^2.
  \end{equation}
  Therefore, if we restrict the order ${\order{A}}$ to the complement $U = \mathbb{P}^2 \setminus \VV(XYZ)$
  we obtain the algebra $\Lambda = A^{(3)}_{(XYZ)}$ generated over $R = \field[(\frac YX)^{\pm},(\frac ZX)^{\pm}]$
  by the elements
  \begin{equation}
    \label{eq:generators-three-veronese}
    1,\quad
    \frac{x^2y}{X},\quad
    \frac{xy^2}{X},\quad
    \frac{x^2z}{X},\quad
    \frac{xz^2}{X},\quad
    \frac{y^2z}{X},\quad
    \frac{yz^2}{X},\quad
    \frac{xyz}{X},\quad
    \frac{x^2y^2z^2}{X^2}.
  \end{equation}
  Each of these generators is invertible:
  for example $\frac{x^2y}{X} \cdot \frac{xy^2}{X} = q_{1,2}^{-1}\frac{Y}{X}$
  and $\frac{xyz}{X} \cdot \frac{x^2y^2z^2}{X^2} = q_{1,2}q_{2,3}q_{1,3} \frac{YZ}X$ are units in $R$,
  and the other cases can be checked similarly.
  Hence if ${\mathfrak{p}} \subset R$ is a prime of codimension $1$
  then the localization $\Lambda_{\mathfrak{p}}$
  has radical $\rad \Lambda_{\mathfrak{p}} = {\mathfrak{p}} \Lambda_{\mathfrak{p}}$
  and ${\order{A}}$ is therefore unramified on $U$.

  Now we check that the order has ramification index~$3$ on each line.
  By symmetry it suffices to consider the line $Z=0$ on the open chart $U = \mathbb{P}^2 \setminus \VV(XY)$.

  The restriction of ${\order{A}}$ to $\{XY\neq 0\}$
  is the algebra $\Lambda = A^{(3)}_{(XY)}$,
  generated as a module over $R=\field[(\frac YX)^\pm,\frac ZX]$
  by the same elements \eqref{eq:generators-three-veronese}.
  We claim that the localization at ${\mathfrak{p}} = (\frac ZX)$
  has radical given by the two-sided ideal
  \begin{equation}
    \rad \Lambda_{\mathfrak{p}} = {\mathfrak{p}}\Lambda_{\mathfrak{p}} + R_{\mathfrak{p}}\left\{\frac{x^2z}{X},\frac{xz^2}{X},\frac{y^2z}{X},\frac{yz^2}{X},\frac{xyz}{X}, \frac{x^2y^2z^2}{X^2}\right\}.
  \end{equation}
  Indeed,
  the generators $1$, $\frac{x^2y}{X}$, $\frac{xy^2}{X}$ are already invertible in $\Lambda$,
  and if $g$ is any other generator in \eqref{eq:generators-three-veronese},
  then~$g^3$ is divisible by $Z=z^3$ and therefore lies in $\rad \Lambda_{\mathfrak{p}}$.
  Moreover, any product $g_1g_2g_3$ of three such generators is divisible by $Z=z^3$, which shows that
  \begin{equation}
    (\rad \Lambda_{\mathfrak{p}})^3 \subset {\mathfrak{p}} \Lambda_{\mathfrak{p}}.
  \end{equation}
  In fact this is an equality, as,
  e.g., $(\frac{xyz}{X})^{3} = \frac{YZ}{X^2}$ is a generator for ${\mathfrak{p}}\Lambda_{\mathfrak{p}}$.
  It follows that $\Lambda$ has ramification index~$3$ on the line $Z=0$, which shows the result.
\end{proof}

\begin{proof}[Proof of \cref{proposition:kanazawa-fermat-cubic-2}]
  Let $B = A/(f)$ where $f=ax^3+by^3+cz^3$ is a noncommutative cubic.
  Then it follows from \Cref{proposition:restriction-central-proj} that there is an equivalence of categories
  \begin{equation}
    \qgr B \simeq \coh(C,\order{B}),
  \end{equation}
  where $C = \VV(aX+bY+cZ) \subset \mathbb{P}^2$ and $\order{B} = \order{A}|_C$.
  Because $\order{A}$ is a tame order, it follows by \cref{theorem:restriction} that
  \begin{equation}
    \coh(C,\order{B}) \simeq \coh(\stack{C}_\rt,\order{B}_\rt) \simeq \coh(\stack{C}_\can,\order{B}_\can),
  \end{equation}
  where $\stack{C}_\rt,$ and $\stack{C}_\can$ are the restriction of the root stack and canonical stack of $\order{A}$ to $C$.
  The first part of the statement therefore follows taking $\stack{C} = \stack{C}_\can$.

  The precise form of this stack depends on the intersection of $C$ with the ramification locus $\Delta$,
  which is of the form $\Delta = \VV(XYZ)$ by \Cref{lemma:nonkanazawa-ramification}.
  We make do a case-by-case analysis.

  Case~(1). If $abc\neq 0$ then the line $C=\VV(aX+bY+cZ)$ meets $\Delta = \VV(XYZ)$ transversely in three points,
  and we can choose an isomorphism~$C\cong \mathbb{P}^1$ so that $C\cap\Delta$ maps to $\{0,1,\infty\}$.
  By \Cref{proposition:restriction-hereditary}\cref{enumerate:restriction-hereditary-root-stack} we therefore obtain the root stack
  \begin{equation}
    \stack{C} = \stack{C}_\can \cong \stack{C}_\rt \cong \sqrt[3]{\mathbb{P}^1; 0+1+\infty}.
  \end{equation}

  Case~(2). If $abc = 0$ but only one of $a,b,c$ is $0$
  then $C=\VV(aX+bY+cZ)$ does not lie in $\Delta$,
  and $\stack{C}_\can \to C$ is an isomorphism outside the two intersection points $C\cap \Delta = \{p_1,p_2\}$,
  where $p_1 \not\in C\cap \Delta_\sing$ and $p_2 \in C\cap \Delta_\sing$.
  By symmetry we may assume that $c=0$.

  By \cref{proposition:r-c-isomorphism}
  the map $\stack{C}_\can \to \stack{C}_\rt$ is an isomorphism over
  the locus~$C\setminus \{p_2\} \cong \mathbb{A}^1$,
  and is therefore (locally) given by the root stack
  \begin{equation}
    \stack{C}_\rt \times_C (C\setminus\{p_2\}) \cong \sqrt[3]{C\setminus\{p_2\}; p_1} \cong \sqrt[3]{\mathbb{A}^1; 0}.
  \end{equation}
  and that $p_2$ is the unique singular point of $\Delta = \VV(XYZ)$ on the chart $\mathbb{P}^2|_{Z\neq 0}$.
  Let $u = \frac{X}{Z}$ and $v=\frac{Y}{Z}$ denote the coordinates on this chart,
  so that the ramification locus is $\VV(uv)\subset\mathbb{A}_{u,v}^2$.
  It follows by \cref{lemma:center-as-scaled-rees-ring,corollary:root-stack-local} that $\stack{C}_\rt$ is given on this chart by
  \begin{equation}
    \stack{S}_\rt \times_{\mathbb{P}^2} \mathbb{P}^2|_{Z\neq 0} \cong \left[\Spec\frac{\field[u,v,t]}{(t^3 - uv)}/\mu_3 \right].
  \end{equation}
  The stack on the right-hand side has an \'etale cover by $\Spec \field[r,s]^{\mu_3}$ where $\mu_3$ acts via $r \mapsto \zeta_3 \cdot r$ and $s \mapsto \zeta_3^2 \cdot s$ for $\zeta_3$ a primitive $3$rd root of unity:
  the map is induced by the ring isomorphism
  \begin{equation}
    \label{equation:invariant-ring-cubic}
    \frac{\field[u,v,t]}{(t^3 - uv)} \xrightarrow{\ \sim\ } \field[r,s]^{\mu_3},\quad
    u\mapsto r^3,\quad
    v\mapsto s^3,\quad
    t\mapsto rs.
  \end{equation}
  The canonical stack can then be described over this same locus via the \'etale cover
  \begin{equation}
    \label{equation:etale-cover-cubic}
    \left[\Spec \field[r,s]/\mu_3 \right] \to \stack{S}_\rt \times_{\mathbb{P}^2} \mathbb{P}^2|_{Z\neq 0}.
  \end{equation}
  In the coordinates $u,v$ the restriction of $C$ to the chart is given by $\VV(au + bv)$,
  so after restricting the canonical stack
  we obtain an \'etale cover of $\stack{C}_\can \times_C (C\setminus \{p_1\})$ of the form
  \begin{equation}
    \left[\Spec \frac{\field[r,s]}{(ar^3+bs^3)}/\mu_3 \right] \longrightarrow \stack{C}_\can \times_C (C\setminus \{p_1\}).
  \end{equation}
  Thus~$\stack{C}_\can$ is singular around the point~$p_2$,
  being the quotient of three concurrent lines.

  Case~(3) If two of $a,b,c$ are $0$, then by symmetry we may assume that $b=c=0$,
  so that $C = \VV(X)$ is a component of the ramification divisor~$\Delta$.
  Writing $U = \mathbb{P}^2|_{Z\neq 0} \cong \Spec \field[u,v]$ with coordinates as above,
  the canonical stack $\stack{S}_\can$ is again described by the \'etale cover in \eqref{equation:etale-cover-cubic}.
  The curve $C$ is cut out by the equation~$u = 0$ on the chart $U$,
  which maps $u\mapsto r^3$ under the isomorphism \eqref{equation:invariant-ring-cubic}.
  The stack $\stack{C}_\can$ can therefore be described by the \'etale cover
  \begin{equation}
    \left[\Spec \frac{\field[r,s]}{(r^3)} / \mu_3\right] \longrightarrow \stack{C}_\can \times_{\mathbb{P}^2} U,
  \end{equation}
  Zariski-locally exhibiting~$\stack{C}_\can$ as
  a non-reduced quotient stack,
  with generic stabiliser~$\mu_3$,
  and special behavior at~$r=0$.
  The description of the other chart $Y\neq 0$ is identical.
\end{proof}

The stacky curve~$\stack{C}$ is a weighted projective line of type~$(3,3,3)$,
which is called \emph{tubular} \cite[\S5.4.2]{MR915180}.
It is one of the (rigid) classes of weighted projective lines with Euler characteristic~0,
thus it shares some properties with cubic curves,
but it is only fractionally Calabi--Yau of dimension~$3/3$.

It is clear from the examples in \cref{section:noncommutative-conics,section:noncommutative-cubics}
that one can study many other examples of central curves on noncommutative surfaces
using these methods.
We illustrate this by giving~2~more examples.

\begin{example}
  \label{example:Clifford-degree-4}
  Consider a graded Clifford algebra~$A$
  as in \cref{section:noncommutative-conics},
  where we will freely use the results obtained in that section.
  Let~$f\in\ZZ(A)_4$
  be a central element of degree~4.
  The space of such elements is spanned by~$x^4,y^4,z^4,x^2y^2,x^2z^2,y^2z^2$,
  and~$f$ defines a conic~$C$ in the projective plane underlying the central Proj.
  Assume that~$C$ is smooth,
  and that~$C\cap\Delta$ is a transverse intersection
  in~6~distinct points.
  Then by \cref{proposition:restriction-hereditary}
  we have an equivalence
  \begin{equation}
    \qgr A/(f)
    \simeq
    \coh\sqrt[2]{\mathbb{P}^1;C\cap\Delta},
  \end{equation}
  where on the right-hand side
  we have a weighted projective line of type~$(2,2,2,2,2,2)$.
  By varying~$f$ we can vary the location of the points
  and obtain a non-isotrivial family of weighted projective lines,
  because we can only choose coordinates to fix~3~of the points.
  This contrasts with the isotrivial family of tubular weighted projective line
  earlier in this subsection.
\end{example}

For a non-isotrivial family of tubular weighted projective lines
arising from our construction,
we turn our attention towards an exotic del Pezzo order on~$\mathbb{P}^2$.
Unlike the previous example,
it does not arise from a central Proj construction
applied to a quadratic Artin--Schelter regular algebra.

\begin{example}
  \label{example:exotic-del-pezzo-order}
  Let~$\order{A}$ be a quaternion order on~$\mathbb{P}^2$
  ramified along a smooth quartic~$\Delta$.
  See, e.g., \cite[\S6]{MR2183391}
  for an explicit construction of such an order using
  the noncommutative cyclic covering trick.
  Let~$L\subset\mathbb{P}^2$ be a line
  such that~$L\cap\Delta$ is a transverse intersection
  in~4~distinct points.
  Then by \cref{proposition:restriction-hereditary}
  we have an equivalence
  \begin{equation}
    \qgr A/(f)
    \simeq
    \coh\sqrt[2]{\mathbb{P}^1;C\cap\Delta},
  \end{equation}
  where on the right-hand side
  we have a weighted projective line of type~$(2,2,2,2)$.
  By varying~$L$ we can vary the location of the points as in the previous example,
  and we obtain a non-isotrivial family of \emph{tubular} weighted projective lines.
  They are fractionally Calabi--Yau of dimension~$2/2$.
\end{example}

\appendix

\section{Generalizing a standard result to stacks}
\label{appendix:lemmas}
We consider a stacky version of the noncommutative schemes considered in \cite{MR4554471}.

\begin{definition}
  A \emph{noncommutative Deligne--Mumford stack} is a pair~$(X,\order{R}_X)$
  of a Deligne--Mumford stack~$X$
  equipped with a quasi-coherent sheaf~$\order{R}_X \in \Qcoh(X)$
  endowed with an algebra structure~$m_X\colon \order{R}_X \otimes \order{R}_X \to \order{R}_X$.
  A morphism between noncommutative stacks~$(X,\order{R}_X)$ and~$(Y,\order{R}_Y)$ is a pair~$\mathbf{f} = (f,f^\sharp)$ of morphisms
  \begin{equation}
    f\colon X \to Y,\quad f^\sharp\colon f^*\order{R}_Y \to \order{R}_X,
  \end{equation}
  with~$f$ a morphism of stacks, and~$f^\sharp$ a morphism of algebra objects in~$\Qcoh(X)$.
\end{definition}

In what follows we will always assume that the all sheaves of algebras are coherent.
With this assumption, any morphism~$\mathbf{f} = (f,f^\sharp) \colon (X,\order{R}_X) \to (Y,\order{R}_Y)$
with~$f$ proper defines pushforward and pullback morphisms
\begin{equation}
  \mathbf{f}_* \colon \coh(X,\order{R}_X) \to \coh(Y,\order{R}_Y),\quad \mathbf{f}^*\colon \coh(Y,\order{R}_Y) \to \coh(X,\order{R}_X),
\end{equation}
where for~$\mathcal{M}\in \coh(X,\order{R}_X)$ the module~$\mathbf{f}_*\mathcal{M}$
is given by~$f_*\mathcal{M}$ endowed with a~$\order{R}_Y$-module structure via
\begin{equation}
  f_*\mathcal{M} \otimes_Y \order{R}_Y
  \to
  f_*(\mathcal{M} \otimes_Y f^*\order{R}_Y) \xrightarrow{f_*(\id_\mathcal{M} \otimes_Y f^\sharp)} f_*(\mathcal{M} \otimes_Y \order{R}_X) \xrightarrow{f_*m_X} f_*\mathcal{M},
\end{equation}
and the pullback of~$\mathcal{N}\in\coh(Y,\order{R}_Y)$
is given by the right~$\order{R}_X$-module~$\mathbf{f}^*\mathcal{N} \colonequals f^*\mathcal{N} \otimes_{f^*\order{R}_Y} \order{R}_X$.
These functors form an adjoint pair as in \cite{MR4554471}.

Given a noncommutative stack~$(X,\order{R}_X)$ and a morphism~$f\colon X\to Y$,
there is an induced noncommutative stack structure on~$Y$ given by~$f_*\order{R}_X$
and a natural map~$\mathbf{f} \colon (X,\order{R}_X) \to (Y,f_*\order{R}_X)$ defined by the counit
\begin{equation}
  f^\sharp \colon f^*f_*\order{R}_X \to \order{R}_X,
\end{equation}
of the adjunction between~$f_*$ and~$f^*$.
Likewise, if~$(Y,\order{R}_Y)$ is a noncommutative stack and~$f\colon X\to Y$ a morphism,
there is an induced noncommutative stack structure on~$X$ defined by~$f^*\order{R}_Y$, and
\begin{equation}
  \mathbf{f} = (f,\id_{f^*\order{R}_Y}) \colon (X,f^*\order{R}_Y) \to (Y,\order{R}_Y)
\end{equation}
is a well-defined map of noncommutative stacks.

\begin{lemma}
  \label{lemma:embedding-as-torsion}
  Let~$(X,\order{R})$ be a noncommutative stack, and~$i\colon V \to X$ a closed immersion.
  Then
  \begin{equation}
    \mathbf{i}_*\colon \coh(V,i^*\order{R}) \hookrightarrow \coh(X,\order{R})
  \end{equation}
  embeds~$\coh(V,i^*\order{R})$ as the full subcategory of modules whose underlying sheaf lies in~$i_*\coh(V)$.
\end{lemma}

\begin{proof}
  Since~$i\colon V\to X$ is a closed embedding of stacks, we claim that the counit~$i^*i_*\to \id$ of the adjunction is a natural isomorphism.
  To see this, we can pull back~$i$ to a morphism~$i'\colon V\times_X U \to U$, where~$U\to X$ is an atlas.
  By \cite[Lemma 04XC]{stacks} the map~$i'$ is a closed immersion of algebraic spaces, and~$(i')^*(i')_*\to \id$ is a natural isomorphism by standard results \cite[Lemma 04CJ]{stacks}.
  Then for any sheaf~$\mathcal{F}$ on~$V$ the morphism~$i^*i_*\mathcal{F} \to \mathcal{F}$ pulls back to the isomorphism~$(i')^*(i')_*\mathcal{F} \to \mathcal{F}$, hence~$i^*i_*\mathcal{F} \to \mathcal{F}$ is also an isomorphism by \cite[Lemma 0GQF]{stacks}.

  For the morphism~$\mathbf{i}$, we note that the left adjoint to~$\mathbf{i}_*$ is simply given by
  \begin{equation}
    \mathbf{i}^*(-) = i^*(-) \otimes_{i^*\order{R}} i^*\order{R} \cong i^*(-),
  \end{equation}
  which agrees with the left adjoint of~$i_*\colon \coh(S) \to \coh(T)$ on the underlying sheaves.
  In particular, for any~$\mathcal{M}\in\coh(V,i^*\order{R})$ the natural map
  \begin{equation}
    \mathbf{i}^*\mathbf{i}_*\mathcal{M} = i^*i_*\mathcal{M} \xrightarrow{\sim} \mathcal{M}
  \end{equation}
  is an isomorphism.
  Hence, the transformation~$\mathbf{i}^*\mathbf{i}_* \to \id$ is a natural isomorphism and hence~$\mathbf{i}_*$ is fully faithful.

  Finally, it is clear by construction that the objects in the image of~$\mathbf{i}_*$ have underlying sheaf in the image of~$i_*$.
  Conversely, if~$\mathcal{M}\in\coh(V,i^*\order{R})$ is a module with underlying sheaf~$i_*\mathcal{F}$ then the natural map~$\mathcal{M} \to \mathbf{i}_*\mathbf{i}^*\mathcal{M} = i_*i^*i_*\mathcal{F} \cong \mathcal{M}$ is an isomorphism, so~$\mathcal{M}$ is the image of~$\mathbf{i}^*\mathcal{M}$.
\end{proof}


\renewcommand*{\bibfont}{\small}
\printbibliography

\emph{Thilo Baumann}, \url{thilo.baumann@uni.lu} \\
Department of Mathematics, Universit\'e de Luxembourg, 6, avenue de la Fonte, L-4364 Esch-sur-Alzette, Luxembourg

\emph{Pieter Belmans}, \url{pieter.belmans@uni.lu} \\
Department of Mathematics, Universit\'e de Luxembourg, 6, avenue de la Fonte, L-4364 Esch-sur-Alzette, Luxembourg

\emph{Okke van Garderen}, \url{okke.vangarderen@sissa.it} \\
Scuola Internazionale Superiore di Studi Avanzati (SISSA), via Bonomea, 265, 34136 Trieste, Italy

\end{document}